\documentclass[11pt]{article}
\usepackage{enumerate}
\usepackage{amssymb,a4wide,latexsym,makeidx,epsfig}
\usepackage{amsthm}
\usepackage{dsfont}
\usepackage{amsmath}
\usepackage{lipsum}
\usepackage{enumerate}
\usepackage{mathrsfs}
\usepackage{xcolor}
\usepackage{tikz}
\usepackage{setspace}
\usepackage{geometry}
\usepackage{appendix}
\usepackage{multirow}
\usepackage[colorlinks=true, linkcolor=black, anchorcolor=black, citecolor=black, urlcolor=black, CJKbookmarks=true]{hyperref}
\usepackage{microtype}
\usepackage{colonequals}
\usepackage{enumitem}
\allowdisplaybreaks

\newtheorem{theorem}{Theorem}[section]

\newtheorem{definition}[theorem]{Definition}
\newtheorem{lemma}[theorem]{Lemma}

\newtheorem{observation}[theorem]{Observation}

\newtheorem{conjecture}[theorem]{Conjecture}
\newtheorem{claim}{Claim}

\newcommand{\ex}{{\rm ex}}
\newcommand{\ext}{{\rm ext}}

\begin{document}
\textwidth 150mm \textheight 225mm

\title{On the multicolor Tur\'{a}n conjecture for color-critical graphs
}
\author{
Xihe Li\footnote{School of Mathematics and Statistics, Shaanxi Normal University, Xi'an, Shaanxi 710119, China.}~~~~~~
Jie Ma\footnote{School of Mathematical Sciences, University of Science and Technology of China, Hefei, Anhui 230026, China.}~\footnote{Yau Mathematical Sciences Center, Tsinghua University, Beijing 100084, China.}~~~~~~
Zhiheng Zheng\footnotemark[2]
}
\date{}
\maketitle
\begin{center}
\begin{minipage}{120mm}
\vskip 0.3cm
\begin{center}
{\small {\bf Abstract}}
\end{center}
{\small A {\it simple $k$-coloring} of a multigraph $G$ is a decomposition of the edge multiset as a disjoint sum of $k$ simple graphs which are referred to as colors.
A subgraph $H$ of a multigraph $G$ is called {\it multicolored} if its edges receive distinct colors in a given simple $k$-coloring of $G$.
In 2004, Keevash-Saks-Sudakov-Verstra\"{e}te introduced the {\it $k$-color Tur\'{a}n number} $\ex_k(n,H)$,
which denotes the maximum number of edges in an $n$-vertex multigraph that has a simple $k$-coloring containing no multicolored copies of $H$.
They made a conjecture for any $r\geq 3$ and $r$-color-critical graph $H$ that in the range of $k\geq \frac{r-1}{r-2}(e(H)-1)$, if $n$ is sufficiently large, then $\ex_k(n, H)$ is achieved by the multigraph consisting of $k$ colors all of which are identical copies of the Tur\'{a}n graph $T_{r-1}(n)$.
In this paper, we show that this holds in the range of $k\geq 2\frac{r-1}{r}(e(H)-1)$, significantly improving earlier results.
Our proof combines the stability argument of Chakraborti-Kim-Lee-Liu-Seo with a novel graph packing technique for embedding multigraphs.
}
\end{minipage}
\end{center}

\section{Introduction}
\label{sec:introduction}

One of the central topics in extremal graph theory is the Tur\'{a}n type problem which asks, for a fixed graph $H$, what is the maximum number of edges in an $H$-free\footnote{Given a graph $H$, we say that a graph is $H$-free if it contains no subgraph isomorphic to $H$.} graph on $n$ vertices.
Such a maximum number is called the {\it Tur\'{a}n number} (or {\it extremal number}) of $H$ and is denoted by $\ex(n,H)$.
The well-known Mantel's theorem \cite{Man907} determines the Tur\'{a}n number for triangles, and Tur\'{a}n's theorem \cite{Tur} generalizes the result from triangles to general cliques.
Let $T_{r-1}(n)$ be the $(r-1)$-partite {\it Tur\'{a}n graph}, that is, the complete $(r-1)$-partite graph on $n$ vertices with part sizes as equal as possible.
Let $t_{r-1}(n)$ be the number of edges of $T_{r-1}(n)$.
Tur\'{a}n's theorem says that $\ex(n, K_r)=t_{r-1}(n)$ which is uniquely attained by $T_{r-1}(n)$.
For a general graph $H$, the celebrated Erd\H{o}s-Stone-Simonovits theorem \cite{ErSi966,ErSt} shows that $\ex(n, H)=\big(1-\frac{1}{\chi(H)-1}+o(1)\big)\frac{n^2}{2}$, where $\chi(H)$ is the chromatic number of $H$.
A graph is called {\it $r$-color-critical} if it has chromatic number $r$, and it has an edge (called a {\it critical edge}) whose removal reduces the chromatic number to $r-1$.
For any $r$-color-critical graph $H$ and sufficiently large $n$, Simonovits \cite{Sim968} proved that $\ex(n, H)=t_{r-1}(n)$ and $T_{r-1}(n)$ is the unique extremal graph.

In the last two decades, the study of extremal problems for multicolor versions became an active research topic.
Let $\mathcal{G}=\{G_1, \ldots, G_k\}$ be a collection of not necessarily distinct graphs on the same vertex set $V$.
A graph $H$ with $V(H)\subseteq V$ is called a {\it multicolored} (or {\it rainbow}) subgraph of $\mathcal{G}$ if there exists an injection $\varphi: E(H)\to [k]$ such that $e\in E(G_{\varphi(e)})$ for each $e\in E(H)$.
The following general question has been extensively studied.

\begin{center}
\begin{minipage}{126mm}
{\it Let $\mathcal{G}=\{G_1, \ldots, G_k\}$ be a collection of graphs on the same vertex set $V$, and $H$ be a graph with $e(H)\leq k$.
Which extremal conditions, when imposed on $\mathcal{G}$, can lead to the existence of a multicolored copy of $H$?}
\end{minipage}
\end{center}
Roughly speaking, two distinct types of extremal conditions have been studied:
\begin{itemize}
\item minimum-type: $\min_{i\in [k]}e(G_i)$ \cite{ADGSMS,BaGr,BaGP,IKLS,KiKu} and $\min_{i\in [k]}\delta(G_i)$ \cite{CHWW,FaMR24+JGT,FaMR24+Combin,JoKi,LuWY23JCTB,LuYY21JCTA};

\item average-type: $\sum_{i\in [k]}e(G_i)$ \cite{BaGP,CKLLS,Fra22+,GGPP,KSSV,MaHo} and $\prod_{i\in [k]}e(G_i)$ \cite{FaMR24+Combin,Fra22+,HFGLSTVZ}.
\end{itemize}
Note that in the special case when all $G_i$ are identical (say $G$), the existence of a multicolored $H$ is equivalent to the existence of a copy of $H$ in $G$.
Thus the study on $\min_{i\in [k]}e(G_i)$ generalizes the original Tur\'{a}n problem.
On the other hand, if $G_1, \ldots, G_k$ are pairwise edge-disjoint matchings, then $\mathcal{G}$ is equivalent to a properly edge-colored graph.
Thus in this case, determining $\sum_{i\in [k]}e(G_i)$ is equivalent to the {\it rainbow Tur\'{a}n problem} introduced by Keevash, Mubayi, Sudakov and Verstra\"{e}te \cite{KMSV}.

In this paper, we mainly focus on the multicolor Tur\'{a}n problem proposed by Keevash, Saks, Sudakov and Verstra\"{e}te \cite{KSSV} in 2004.
The authors of \cite{KSSV} used a different notion rather than the collection of graphs.
A {\it simple $k$-coloring} of a multigraph $G$ is a decomposition of the edge multiset as a disjoint sum of $k$ simple graphs which are referred to as {\it colors}.
A multigraph with a simple $k$-coloring is called a {\it simply $k$-colored multigraph}.
A subgraph\footnote{Throughout this paper, when we refer to a subgraph of a multigraph, it is implied that the subgraph is also a multigraph.} $H$ of a multigraph $G$ is called {\it multicolored} if its edges receive distinct colors in a given simple $k$-coloring of $G$.
The {\it $k$-color Tur\'{a}n number}, denoted by $\ex_k(n,H)$, is the maximum number of edges in an $n$-vertex multigraph that has a simple $k$-coloring containing no multicolored copies of $H$.
The simply $k$-colored multigraphs without multicolored copies of $H$ that achieve this maximum are called the {\it $k$-color extremal multigraphs} of $H$.

If $k \leq e(H)-1$, there are no multicolored copies of $H$ with $k$ colors.
In this case, the unique $k$-color extremal multigraph is the multigraph consisting of $k$ copies of complete graphs.
For $k\geq e(H)$, there are two natural candidates for the $k$-color extremal multigraphs on $n$ vertices:
\begin{itemize}
\item[(i)] the multigraph consisting of $e(H)-1$ copies of the complete graph $K_n$;
\item[(ii)] the multigraph consisting of $k$ identical copies of a fixed extremal $H$-free graph $G_{\ext}$.
\end{itemize}
We use $(e(H)-1)\cdot K_n$ to denote the first multigraph, and $k \cdot G_{\ext}$ to denote the second multigraph.
Keevash et al. \cite{KSSV} proved that the second construction is always extremal when $k$ is sufficiently large.

\begin{theorem}{\normalfont (\cite[Theorem~1.1]{KSSV})}\label{th:KSSV-large_k}
Let $H$ be a graph and $k, n$ be two positive integers.
If $$k\geq {n \choose 2} -\ex(n,H)+ e(H),$$ then $\ex_k(n, H)=k\cdot \ex(n,H)$, and in any $k$-color extremal multigraph of $H$, all the $k$ colors are identical copies of an extremal $H$-free graph.
\end{theorem}

The authors of \cite{KSSV} made the following general conjecture for color-critical graphs.

\begin{conjecture}{\normalfont (\cite[Conjecture~1.3]{KSSV})}\label{conj:KSSV-critical}
Let $r\geq 3$, $k\geq h$, and $H$ be an $r$-color-critical graph with $h$ edges.
Then, for sufficiently large $n$, the $n$-vertex $k$-color extremal multigraph of $H$ either consists of $k$ colors all of which are identical copies of $T_{r-1}(n)$ or consists of exactly $h-1$ colors all of which are copies of $K_n$.
In particular,
$$\ex_k(n, H)=
\left\{
   \begin{aligned}
    &k\cdot t_{r-1}(n) & & \mbox{for $k\geq \frac{r-1}{r-2}(h-1)$},\\
    &\left(h-1\right){n\choose 2} & & \mbox{for $h\leq k< \frac{r-1}{r-2}(h-1)$}.
   \end{aligned}
   \right.$$
\end{conjecture}

We should refer to the range $k\geq \frac{r-1}{r-2}(h-1)$ as the {\it upper range} and the range $h\leq k< \frac{r-1}{r-2}(h-1)$ as the {\it lower range}.
The authors of \cite{KSSV} confirmed this for all cliques $K_r$ when $n>10^{4}r^{34}$.
Recently, Frankl \cite{Fra22+} and Ma and Hou \cite{MaHo} proved that for all $n\geq r-1$, $\ex_k(n, K_r)\leq \max\big\{k\cdot t_{r-1}(n), \left({r\choose 2}-1\right){n\choose 2}\big\}$ holds in the cases $r=3$ and $r=4,5$, respectively.
Conjecture~\ref{conj:KSSV-critical} was confirmed in full for $r=3$ by Keevash et al. \cite{KSSV} and for $r=4$ by Chakraborti, Kim, Lee, Liu and Seo \cite{CKLLS}.
For $r\geq 5$, Chakraborti et al. \cite{CKLLS} proved that Conjecture~\ref{conj:KSSV-critical} holds for a rich family of $r$-color-critical graphs whose edges are distributed somewhat evenly.
Note that Theorem~\ref{th:KSSV-large_k} implies that Conjecture~\ref{conj:KSSV-critical} holds for $k\geq (1+o(1))\frac{1}{r-1}{n\choose 2}$.

In this paper, we show that Conjecture~\ref{conj:KSSV-critical} holds in the range of $k\geq 2\frac{r-1}{r}(h-1)$.

\begin{theorem}\label{th:main}
Let $r\geq 5$ and $H$ be an $r$-color-critical graph with $h$ edges.
If $n$ is sufficiently large and $$k\geq 2\frac{r-1}{r}(h-1),$$ then $\ex_k(n, H)=k\cdot t_{r-1}(n)$, and the unique $n$-vertex $k$-color extremal multigraph of $H$ consists of $k$ colors all of which are identical copies of $T_{r-1}(n)$.
\end{theorem}

In fact, we prove a slightly stronger result (see Theorem~\ref{th:main+}), where the same statement holds under a weaker condition that
$$k> \max\left\{\frac{r-1}{r-2}(h-1), 2\frac{r-1}{r}\left(h-\frac{r(r-2)}{4}+1\right)\right\}.$$
This improves Theorem~\ref{th:KSSV-large_k} from $k\geq (1+o(1))\frac{1}{r-1}{n\choose 2}$ to $k=\Omega(h)$ for all $r$-color-critical graphs.
Our proof combines the stability argument of Chakraborti, Kim, Lee, Liu and Seo \cite{CKLLS} with a novel graph packing technique for embedding multigraphs.
For an illustration of the proof sketch, we refer readers to Section~\ref{sec:proof_sketch}.

The remainder of this paper is organized as follows.
In the next section, we introduce some additional terminology and notation, and state some existing results that will be used in our proofs.
In Section~\ref{sec:proof_sketch}, we supply a proof sketch of our main result.
In Section~\ref{sec:r_vertex}, we show the first step of our proof, that is, establishing an analogous statement for the so-called $r$-vertex $r$-color-critical multigraphs.
In Section~\ref{sec:any_vertex}, we complete our proof of Theorem~\ref{th:main}.
Finally, we conclude the paper with some remarks and open problems in Section~\ref{sec:ch-conclu};
in particular, we prove a tight $k$-colored Tur\'an type result for general graphs with $h$ edges and chromatic number $r$ when $h\leq k\leq h+\big\lfloor\frac{r}{2}\big\rfloor-1$.

\section{Preliminaries}
\label{sec:pre}

In this section, we introduce some additional terminology and notation, as well as several lemmas that will be used in our proof of the main result.

\subsection{Terminology and some technical inequalities}
\label{subsec:def}

Throughout this paper, we consider both simple graphs and multigraphs (a multigraph refers to a multigraph with no loops).
We will also view a simple graph as a multigraph in which every edge has multiplicity 1.

Let $G$ be a multigraph with vertex set $V(G)$ and edge multiset $E(G)$.
Let ${V(G)\choose 2}\colonequals \left\{\{u,v\}\colon\, u,v\in V(G), u\neq v\right\}$.
For any $\{u,v\}\in {V(G)\choose 2}$, we also write it as $uv$ or $vu$.
We shall call an element $e$ in ${V(G)\choose 2}$ an edge, although the multiplicity of $e$ could be 0 in $G$ (i.e., $e$ is in fact not an actual edge of $G$).
For $e\in {V(G)\choose 2}$, the multiplicity of $e$ in $G$ is written as $w_G(e)$.
For a subset $E\subseteq {V(G)\choose 2}$, let $w_G(E)\colonequals \sum_{e\in E}w_G(e)$.
For any $v\in V(G)$ and $U\subseteq V(G)$, let $d_U(v)\colonequals \sum_{u\in U}w_G(vu)$.
The degree $d_G(v)$ of a vertex $v$ in $G$ is the number of edges (counted with multiplicity) incident with $v$, i.e., $d_G(v)=\sum_{u\in V(G)}w_G(vu)$.
Let $\delta(G)\colonequals \min\left\{d_G(v)\colon\, v\in V(G)\right\}$ be the minimum degree of $G$.
We use $e(G)$ to denote the number of edges (counted with multiplicity) of $G$, i.e., $e(G)=|E(G)|=\sum_{e\in {V(G)\choose 2}}w_G(e)$.

Given two disjoint vertex sets $U, V\subseteq V(G)$, let $E(U,V)$ be the set of edges between $U$ and $V$ in $G$, and let $e(U, V)\colonequals |E(U,V)|$.
For a vertex subset $U\subseteq V(G)$, the subgraph of $G$ {\it induced by $U$}, denoted by $G[U]$, is the subgraph of $G$ with vertex set $U$ and edge multiset $\{e\in E(G)\colon\, \mbox{both end-vertices of $e$ are contained in $U$}\}$.
Moreover, let $G-U\colonequals G[V(G)\setminus U]$.
For an edge sub-multiset $E\subseteq E(G)$, let $G[E]$ be the {\it edge-induced} subgraph of $G$ with vertex set $\{v\in V(G)\colon\, \mbox{$v$ is incident with some $e\in E$}\}$ and edge multiset $E$.

For a positive integer $n$, let $[n]\colonequals \{1, 2, \ldots, n\}$.
For a set $S$ of real numbers and $1\leq \ell \leq |S|$, let $\min_{\ell} S$ (resp., $\max_{\ell} S$) be the $\ell$-th smallest (resp., largest) number in $S$.
For the sake of clarity of presentation, we systematically omit floor and ceiling signs whenever they are not crucial.
When we say a result holds for $0\leq a\ll b\leq 1$, we mean that there exists a non-decreasing function $f$ such that the result holds whenever $a\leq f(b)$.
Similarly, when we say a result holds for $0\leq a\ll b,c\ll d\leq 1$, we mean that there exist non-decreasing functions $f_1$ and $f_2$ such that the result holds whenever $a\leq f_1(b,c)$ and $b,c\leq f_2(d)$.
Such hierarchies with more constants are defined in a similar way and are to be read from the right to the left.
We will be using the following standard estimates on Tur\'{a}n graphs (for $n\geq r\geq 2$):
\begin{equation}\label{eq:t_r-1}
\frac{r-2}{r-1}{n\choose 2} \leq \ex(n, K_r)=t_{r-1}(n) \leq \frac{r-2}{r-1}\cdot \frac{n^2}{2}= \frac{r-2}{r-1}{n\choose 2}+\frac{r-2}{2(r-1)}n
\end{equation}
and
\begin{equation}\label{eq:degree_T_r-1}
\frac{r-2}{r-1}(n-1)\leq \delta(T_{r-1}(n))\leq \frac{r-2}{r-1}n.
\end{equation}

We shall also use the following two technical lemmas, and we postpone the proofs to Appendix~\ref{ap:1}.

\begin{lemma}\label{le:calculation-1}
Let $r\geq 5$, $s=\left\lceil\frac{r}{2}\right\rceil$ and $h\geq {r\choose 2}$.
If $\ell$ and $i$ are two integers satisfying one of the following statements:
\begin{itemize}
\item[{\rm (i)}] $2\leq \ell \leq s-2$ and $r-s+1 \leq i\leq r-1$, or
\item[{\rm (ii)}] $\ell = s-1$ and $r-s+2 \leq i\leq r-1$,
\end{itemize}
then $\frac{2}{r}\big(r-2-\frac{\ell-1}{i-\ell}\big)\big(h-\frac{r(r-2)}{4}+1\big)\geq h-((\ell-1)s+r-i).$
\end{lemma}

\begin{lemma}\label{le:calculation-2}
Let $r\geq 5$, $s=\left\lceil\frac{r}{2}\right\rceil$, $2\leq \ell \leq s-1$ and $h\geq {r\choose 2}$.
Then $$2\frac{r-\ell-1}{r-\ell}\cdot \frac{r-1}{r}\left(h-\frac{r(r-2)}{4}+1\right)\geq h-(\ell-1)s.$$
\end{lemma}

\subsection{Graph packing}
\label{subsec:packing}

A {\it packing} of two graphs $G$ and $H$ is a bijection $\sigma: V(G)\to V(H)$ where $uv\in E(G)$ implies $\sigma(u)\sigma(v)\notin E(H)$.
In other words, there is a packing of two graphs $G$ and $H$ if and only if $G\subseteq \overline{H}$.
We shall use the following result of Sauer and Spencer \cite{SaSp}.

\begin{theorem}{\normalfont (\cite[Theorem~2]{SaSp})}\label{th:packing-prod}
Let $G$ and $H$ be two graphs with $n$ vertices.
If $e(G)e(H)<{n\choose 2}$, then there is a packing of $G$ and $H$.
\end{theorem}

The following result was conjectured by Milner and Welsh \cite{MiWe}, and proved by Sauer and Spencer \cite{SaSp} and Bollob\'{a}s and Eldridge \cite{BoEl} independently; see also \cite[Corollary~3.3]{Bra995}.

\begin{theorem}[\cite{BoEl,Bra995,SaSp}]\label{th:packing-sum}
Let $G$ and $H$ be two graphs with $n$ vertices.
If $e(G)+e(H)\leq \frac{3}{2}(n-1)$, then there is a packing of $G$ and $H$.
\end{theorem}

\subsection{Existing results for multicolor Tur\'an problems}
\label{subsec:tool}
In this subsection, we state several existing results from \cite{CKLLS,KSSV} for the use of our proof.
The {\it underlying graph} of a multigraph $H$ is the simple graph with vertex $V(H)$ and edge set $\big\{e\in {V(H)\choose 2}\colon\, w_H(e)\geq 1\big\}$.
The chromatic number of a multigraph is given by the chromatic number of its underlying graph.
A multigraph is called {\it $r$-color-critical} if it has chromatic number $r$, and it has an edge (i.e., a {\it critical edge}) whose removal decreases the chromatic number.
We say that a simple $k$-coloring is {\it nested} if its colors form a chain under inclusion, i.e., a simply $k$-colored multigraph with colors $G_1, \ldots, G_k$ is nested if $G_{\pi(1)}\subseteq \cdots \subseteq G_{\pi(k)}$ for some permutation $\pi$ on $[k]$.
It was shown in \cite{CKLLS,KSSV} that if $G$ is a simply $k$-colored multicolored-$H$-free multigraph, then there is a simply $k$-nested-colored multicolored-$H$-free multigraph $F$ with $V(F)=V(G)$ and $E(F)=E(G)$.
The following lemma reduces the upper range of Conjecture~\ref{conj:KSSV-critical} to nested multigraphs with high minimum degree.
This lemma is in fact a consequence of both \cite[Proposition~3.3]{CKLLS} (an analogous idea was also used in the proofs of Theorems~3.1 and 3.2 in \cite{KSSV}) and \cite[Proposition~3.5]{CKLLS} (see also \cite[Lemma~2.1]{KSSV}).

\begin{lemma}[\cite{CKLLS,KSSV}]\label{le:nested_degree}
Let $r\geq 3$, $k\geq 1$, and $H$ be an $r$-color-critical multigraph.
Suppose that there exists an $n_0$ such that for all $n\geq n_0$, every $n$-vertex simply $k$-nested-colored multicolored-$H$-free multigraph $G$ with $e(G)\geq k\cdot t_{r-1}(n)$ and $\delta(G)\geq k\delta(T_{r-1}(n))$ must be a $k\cdot T_{r-1}(n)$.
Then there exists an $n_1$ such that for all $n\geq n_1$, every $n$-vertex simply $k$-colored multicolored-$H$-free multigraph $G$ with $e(G)\geq k\cdot t_{r-1}(n)$ must be a $k\cdot T_{r-1}(n)$.
\end{lemma}

To show a host multigraph under certain circumstances is $(r-1)$-partite,
the following family of subgraphs with specified property serves as an intermediate step in our proof.

\begin{definition}{\normalfont (\cite[Definition~3.7]{CKLLS})}\label{def:H-friendly}
{\rm Let $k\geq \frac{r-1}{r-2}(h-1)$, $H$ be an $r$-color-critical multigraph with $h$ edges, and $K$ be a simply $k$-colored $(r-1)$-partite multigraph with partite sets $W_1, \ldots, W_{r-1}$ of equal size $t$.
We say that $K$ is {\it $H$-friendly}\footnote{In \cite{CKLLS}, the concept of an $H$-friendly multigraph was also defined for the lower range $k<\frac{r-1}{r-2}(h-1)$.} if the multigraph obtained in the following way always contains a multicolored $H$:
add a new vertex $v$ to $K$ and add edges incident to $v$ with multiplicity at most $k$ so that $\sum_{i\in [r-1]}d_{W_i}(v)\geq (r-2)tk$ and $d_{W_i}(v)\geq 1$ for all $i\in [r-1]$.}
\end{definition}

The following lemma, proved in \cite{CKLLS} using a stability argument,
states that if there exists an $H$-friendly induced subgraph in a simply $k$-colored multigraph $G$ with certain properties, then $G$ possesses the desired $(r-1)$-partite global structure.

\begin{lemma}{\normalfont (\cite[Lemma~3.8]{CKLLS})}\label{le:partite}
Let $0<\frac{1}{n}\ll \delta \ll \frac{1}{k}, \frac{1}{m}, \frac{1}{t} \leq 1$, $r\geq 4$ and $k\geq \frac{r-1}{r-2}(h-1)$.
Let $H$ be an $m$-vertex $r$-color-critical multigraph with $h$ edges, and $G$ be an $n$-vertex simply $k$-colored multicolored-$H$-free multigraph with $\delta (G)\geq (1-\delta)k\delta(T_{r-1}(n))$.
If $G$ contains a $t(r-1)$-vertex $H$-friendly subgraph as an induced subgraph, then $G$ is $(r-1)$-partite.
\end{lemma}

We will use the following result from \cite{CKLLS} iteratively to find a subset of vertices such that, there are many edges with high multiplicities incident with them.
The same idea was also employed in \cite{KSSV} (see, for example, the proof of \cite[Theorem~3.2]{KSSV}).

\begin{lemma}{\normalfont (\cite[Proposition~3.4]{CKLLS})}\label{le:vertices}
Let $0<\frac{1}{n}\ll \delta \ll \frac{1}{k}, \frac{1}{d}, \frac{1}{t} <1$.
Suppose that $G$ is an $n$-vertex simply $k$-colored multigraph with $\delta (G)\geq (1-\delta)d(n-1)$, and $U\subseteq V(G)$ is a vertex subset of size $t$.
Then there is a vertex $v\in V(G)\setminus U$ such that $d_U(v)\geq dt$.
\end{lemma}

Given two multigraphs $H$ and $G$, we say that an injection $\phi: V(H) \to V(G)$ is an {\it embedding} if $w_G(\phi(u)\phi(v))\geq w_H(uv)$ for all $uv\in {V(H)\choose 2}$.
For $e=uv\in {V(H)\choose 2}$, we also write $\phi(e)$ for $\phi(u)\phi(v)$.
The following lemma provides a sufficient condition that guarantees the existence of a multicolored subgraph.
An intuitive explanation for this lemma is that we can greedily select colors of edges with respect to the ordering $(f_1, \ldots, f_t)$.

\begin{lemma}{\normalfont (\cite[Proposition~3.6]{CKLLS})}\label{le:Hall}
Let $G$ be a simply $k$-colored multigraph, $H$ be a multigraph, and $t=\big|\big\{e\in {V(H)\choose 2}\colon\, w_H(e)\geq 1 \big\}\big|$.
Suppose that there is an embedding $\phi: V(H) \to V(G)$.
If there is an enumeration $(f_1, \ldots, f_t)$ of the edges in $\big\{e\in {V(H)\choose 2}\colon\, w_H(e)\geq 1\big\}$ such that $w_G(\phi(f_i))\geq \sum_{\ell=1}^{i}w_H(f_{\ell})$ for all $i \in [t]$, then $\phi(H)$ yields a multicolored copy of $H$ in $G$.
\end{lemma}

\section{Proof sketch of Theorem~\ref{th:main}}
\label{sec:proof_sketch}

In this section, we supply a proof sketch of Theorem~\ref{th:main}.
In fact, we will prove Theorem~\ref{th:main} in a slightly stronger form (Theorem~\ref{th:main+}).
To illustrate the sketch of our proof, we need some additional definitions introduced in \cite{CKLLS}.
For an $r$-color-critical multigraph $H$, we call a proper vertex-coloring $c$ with color classes $V_1, \ldots, V_r$ a {\it critical coloring} if there exists two colors $i,j$ such that $e(V_i, V_j)=1$.
Given its critical coloring $c$, the {\it color-reduced multigraph} of $H$, denoted as $H_{c}$, is the multigraph with vertex set $[r]$ and $w_{H_c}(ij)=e(V_i,V_j)$ for each $ij\in {[r]\choose 2}$.
Note that $H_c$ is an $r$-vertex $r$-color-critical multigraph.
Intuitively, if we can prove a multicolor Tur\'{a}n-type result for such a multigraph $H_c$, then we can deal with the reduced multigraphs obtained by applying the multicolor Regularity Lemma to multicolored-$H$-free multigraphs, and finally deduce our main result using this intermediate result and the Embedding Lemma.

In our proof, expanding on the stability argument proposed by Chakraborti, Kim, Lee, Liu and Seo \cite{CKLLS}, we introduce a graph packing technique for embedding multigraphs.
Our proof consists of two steps as follows:
%
\begin{itemize}[leftmargin=0.55cm, itemindent=1cm]
\item[{\bf Step I.}] Establishing the result for $r$-vertex $r$-color-critical multigraphs (see Theorem~\ref{th:r_vertex}) via graph packing arguments; \vspace{-0.2cm}
\end{itemize}
\begin{itemize}[leftmargin=0.7cm, itemindent=1cm]
\item[{\bf Step II.}] Completing the proof of Theorem~\ref{th:main+} (for general $r$-color-critical graphs) via stability arguments. \vspace{-0.1cm}
\end{itemize}

In Step I, the key is to prove Lemma~\ref{le:r_vertex-degree}, which states that for an $r$-vertex $r$-color-critical multigraph $H$,
if a host multigraph with relatively large minimum degree contains no multicolored $H$, then it must be an $(r-1)$-partite multigraph.
Once this lemma is proven, we can readily deduce Theorem~\ref{th:r_vertex}.
This theorem represents a result that is analogous to Theorem~\ref{th:main+}, but it specifically pertains to $r$-vertex $r$-color-critical multigraphs.
In the proof of Lemma~\ref{le:r_vertex-degree}, our main task is to find a multicolored copy of $H$ in certain $r$-vertex multigraph $F^+$.
The multigraph $F^+$ contains a subset $E_1$ of edges with low multiplicities, while the multigraph $H$ contains a subset $E_2$ of edges with high multiplicities.
When embedding $H$ into $F^+$, one must embed edges of low multiplicities in $H$ into the edge-set $E_1$.
To this end, it suffices to show that there is a packing of $F^+[E_1]$ and $H[E_2]$ (we may add some isolated vertices if these two graphs have different orders),
which we manage to achieve by carefully estimating the multiplicities of the relevant edges.

In Step II, we complete the proof of Theorem~\ref{th:main+} (and also Theorem~\ref{th:main}) by a routine application of the stability argument,
which in this context was introduced in \cite{CKLLS}.
In order to extend the result from $r$-vertex critical multigraphs to critical graphs with any number of vertices,
we apply the multicolor version of Szemer\'{e}di's Regularity Lemma to a multicolored-$H$-free multigraph, and get a reduced multigraph.
By the Embedding Lemma, we argue this reduced multigraph must be multicolored-$H_c$-free, where $H_c$ is an $r$-vertex $r$-color-critical multigraph.
Finally, we make use of Theorem~\ref{th:r_vertex} given by Step I to derive the main result of this paper.


\section{Step I: graph packing of $r$-vertex multigraphs}
\label{sec:r_vertex}

In this section, we prove the following result for $r$-vertex $r$-color-critical multigraphs.

\begin{theorem}\label{th:r_vertex}
Let $r\geq 5$ and $H$ be an $r$-vertex $r$-color-critical multigraph with $h$ edges.
If $n$ is sufficiently large and $$k> \max\left\{\frac{r-1}{r-2}(h-1), 2\frac{r-1}{r}\left(h-\frac{r(r-2)}{4}+1\right)\right\},$$ then $\ex_k(n, H)=k\cdot t_{r-1}(n)$,
and the unique $n$-vertex $k$-color extremal multigraph of $H$ consists of $k$ colors all of which are identical copies of $T_{r-1}(n)$.
\end{theorem}

The following lemma is crucial in the proofs of Theorem~\ref{th:r_vertex} and a technical lemma in Section~\ref{sec:any_vertex} (i.e., Lemma~\ref{le:r_vertex-stability}).

\begin{lemma}\label{le:r_vertex-degree}
Let $r\geq 5$, $k>\max\big\{\frac{r-1}{r-2}(h-1), 2\frac{r-1}{r}\big(h-\frac{r(r-2)}{4}+1\big)\big\}$ and $0 \ll \frac{1}{n} \ll \delta \ll \frac{1}{k}$.
Let $H$ be an $r$-vertex $r$-color-critical multigraph with $h$ edges, and $G$ be an $n$-vertex simply $k$-colored multicolored-$H$-free multigraph with $\delta(G)\geq (1-\delta)k\delta(T_{r-1}(n))$. Then $G$ is $(r-1)$-partite.
\end{lemma}

\noindent {\bf Proof of Theorem~\ref{th:r_vertex} (assuming Lemma~\ref{le:r_vertex-degree}).} Let $G$ be an $n$-vertex simply $k$-colored multicolored-$H$-free multigraph with $e(G)\geq k\cdot t_{r-1}(n)$.
We shall show that $G=k\cdot T_{r-1}(n)$.
By Lemma~\ref{le:nested_degree}, we may assume that $\delta(G)\geq k\delta(T_{r-1}(n))$.
By Lemma~\ref{le:r_vertex-degree}, $G$ is $(r-1)$-partite.
Since $T_{r-1}(n)$ is the unique $(r-1)$-partite graph on $n$ vertices with $t_{r-1}(n)$ edges, we have $e(G)\leq k\cdot t_{r-1}(n)$.
Thus $e(G)= k\cdot t_{r-1}(n)$ and $G=k\cdot T_{r-1}(n)$.
The proof is complete.
\hfill$\square$
\vspace{0.2cm}

In the proof of Lemma~\ref{le:r_vertex-degree}, we will repeatedly use the following simple observation.

\begin{observation}\label{ob:r vertex}
Let $H$ be an $r$-vertex $r$-color-critical multigraph with $h$ edges. Then
\begin{itemize}
\item[{\rm (i)}] $w_H(e)\geq 1$ for every $e\in {V(H)\choose 2}$,
\item[{\rm (ii)}] $w_H(E)\leq h-\left({r\choose 2}-|E|\right)$ for any $E\subseteq {V(H)\choose 2}$, and
\item[{\rm (iii)}] $d_H(v)\leq h-{r-1 \choose 2}$ for any vertex $v\in V(H)$.
\end{itemize}
\end{observation}

We now devote the rest of this section to the proof of Lemma~\ref{le:r_vertex-degree}.
\vspace{0.2cm}

\noindent {\bf Proof of Lemma~\ref{le:r_vertex-degree}.} Since $\delta(G)\geq (1-\delta)k\delta(T_{r-1}(n))$, we have $\delta(G)\geq (1-\delta)k\frac{r-2}{r-1}(n-1)$ by Inequality~(\ref{eq:degree_T_r-1}).
Thus $e(G)\geq \frac{1}{2}n\delta(G)\geq \frac{1}{2}n(1-\delta )k\frac{r-2}{r-1}(n-1)=(1-\delta)k \frac{r-2}{r-1}\binom{n}{2}.$
Since $H$ is an $r$-vertex $r$-color-critical multigraph, we have $h\geq {r\choose 2}$ by Observation~\ref{ob:r vertex} (i).
Let $s\colonequals \left\lceil\frac{r}{2}\right\rceil$.
Let $K_{r-s}^{(h-r(r-2)/4+1)}$ be a complete multigraph on $r-s$ vertices in which every edge has multiplicity at least $h-\frac{r(r-2)}{4}+1$.

\begin{claim}\label{cl:Kr-s}
$G$ contains $K_{r-s}^{(h-r(r-2)/4+1)}$ as a subgraph.
\end{claim}

\begin{proof}
Let $X\colonequals \big\{e\in{V(G)\choose 2}\colon\, w_G(e)\geq h-\frac{r(r-2)}{4}+1\big\}$ and $x\colonequals |X|$.
It suffices to show that $x>\ex(n,K_{r-s})$.
Note that the multiplicity of an edge is a nonnegative integer, so for any $e\in {V(G)\choose 2}\setminus X$, we have $w_G(e)\leq h-\frac{r(r-2)}{4}+1-\frac{1}{4}=h-\frac{r^2-2r-3}{4}$.
Then
\begin{equation*}
(1-\delta)k \frac{r-2}{r-1}\binom{n}{2}\leq e(G)\leq kx+\left(h-\frac{r^2-2r-3}{4}\right)\left(\binom{n}{2}-x\right),
\end{equation*}
and thus
\begin{equation*}
x\geq \frac{1}{k-(h-(r^2-2r-3)/4)}\left((1-\delta)k\frac{r-2}{r-1}-\left(h-\frac{r^2-2r-3}{4}\right)\right)\binom{n}{2}.
\end{equation*}
Since $k>2\frac{r-1}{r}\left(h-\frac{r(r-2)}{4}+1\right)$, $0 \ll \frac{1}{n} \ll \delta \ll \frac{1}{k}$ and $s=\left\lceil\frac{r}{2}\right\rceil \geq \frac{r}{2}$, we have
\begin{equation*}
 \begin{aligned}
  &~(1-\delta)k\frac{r-2}{r-1}-\left(h-\frac{r^2-2r-3}{4}\right)-\left(k-\left(h-\frac{r^2-2r-3}{4}\right)\right)\frac{r-s-2}{r-s-1}\\
  =&~\frac{s}{(r-1)(r-s-1)}k-\delta\frac{r-2}{r-1}k-\frac{1}{r-s-1}\left(h-\frac{r^2-2r-3}{4}\right)\\
  > &~\frac{r/2}{(r-1)(r-s-1)}2\frac{r-1}{r}\left(h-\frac{r(r-2)}{4}+1\right)- \delta\frac{r-2}{r-1}k- \frac{1}{r-s-1}\left(h-\frac{r^2-2r-3}{4}\right)\\
  =&~\frac{1}{4(r-s-1)}-\delta\frac{r-2}{r-1}k ~\geq \frac{1}{5(r-s-1)}.
 \end{aligned}
\end{equation*}
Thus $x\geq \frac{r-s-2}{r-s-1}\binom{n}{2}+\frac{1}{5(r-s-1)(k-(h-(r^2-2r-3)/4))}{n\choose 2}$.
Note that $r-s\geq r-\frac{r+1}{2}\geq 2$ since $r\geq 5$.
By Inequality~(\ref{eq:t_r-1}), we have $\ex(n,K_{r-s})=\frac{r-s-2}{r-s-1}\binom{n}{2}+O(n)$.
Then $x>\ex(n,K_{r-s})$.
The result follows.
\end{proof}

By Claim \ref{cl:Kr-s}, we may assume that $K$ is a copy of $K_{r-s}^{(h-r(r-2)/4+1)}$ in $G$ with $V(K)=\{v_1, v_2, \ldots, v_{r-s}\}$.
Recall that $\delta(G)\geq (1-\delta)k\frac{r-2}{r-1}(n-1)$.
Using Lemma~\ref{le:vertices} iteratively, we can find vertices $v_{r-s+1}, \ldots, v_{r-1}\in V(G)\setminus V(K)$ such that for each $r-s+1 \leq i \leq r-1$, we have $\sum_{j\in [i-1]} w_G(v_iv_j)\geq k\frac{r-2}{r-1}(i-1).$
Since $k>\max\big\{\frac{r-1}{r-2}(h-1), 2\frac{r-1}{r}\big(h-\frac{r(r-2)}{4}+1\big)\big\}$, for each $r-s+1\leq i\leq r-1$ and $2\leq \ell\leq i-2$, we have
\begin{align}\label{eq:min_i}
 ~&~\min \{w_G(v_iv_j)\colon\, 1\leq j\leq i-1\} \nonumber\\
 \geq &~\sum\nolimits_{j\in[i-1]}w_G(v_iv_j)-(i-2)k ~\geq k\frac{r-2}{r-1}(i-1)-(i-2)k ~= k\frac{r-i}{r-1} \nonumber\\
 >&~2\frac{r-1}{r}\left(h-\frac{r(r-2)}{4}+1\right)\frac{r-i}{r-1} ~= (r-i)\frac{2}{r}\left(h-\frac{r(r-2)}{4}+1\right),
\end{align}
\begin{align}\label{eq:min2_i}
 ~&~\min\nolimits_2\{w_G(v_iv_j)\colon\, 1\leq j\leq i-1\} \nonumber\\
 \geq &~\frac{1}{2}\left(\sum\nolimits_{j\in[i-1]}w_G(v_iv_j)-(i-3)k\right) ~\geq \frac{1}{2}\left(k\frac{r-2}{r-1}(i-1)-(i-3)k\right) \nonumber\\
 = &~\frac{1}{2}k\frac{2r-i-1}{r-1} ~> \frac{2r-i-1}{r}\left(h-\frac{r(r-2)}{4}+1\right) \nonumber\\
 \geq &~\frac{2r-(r-1)-1}{r}\left(h-\frac{r(r-2)}{4}+1\right) ~= h-\frac{r(r-2)}{4}+1,
\end{align}
\begin{align}\label{eq:max_i}
 ~&~\max\{w_G(v_iv_j)\colon\, 1\leq j\leq i-1\} \nonumber\\
 \geq &~\frac{1}{i-1}\sum\nolimits_{j\in[i-1]}w_G(v_iv_j) ~\geq \frac{1}{i-1}k\frac{r-2}{r-1}(i-1) ~> \frac{r-1}{r-2}(h-1)\frac{r-2}{r-1} ~= h-1,
\end{align}
and
\begin{align}\label{eq:maxl_i}
 ~&~\max\nolimits_{\ell}\{w_G(v_iv_j)\colon\, 1\leq j\leq i-1\}\nonumber\\
 \geq &~\frac{1}{i-\ell}\left(\sum\nolimits_{j\in[i-1]}w_G(v_iv_j)-(\ell-1)k\right) ~\geq \frac{1}{i-\ell}\left(k\frac{r-2}{r-1}(i-1)-(\ell-1)k\right) \nonumber\\
 = &~k\frac{1}{i-\ell}\cdot \frac{(r-2)(i-1)-(r-1)(\ell-1)}{r-1} ~= k\left(\frac{r-2}{r-1}-\frac{\ell-1}{(r-1)(i-\ell)}\right) \nonumber\\
 > &~2\frac{r-1}{r}\left(h-\frac{r(r-2)}{4}+1\right)\left(\frac{r-2}{r-1}-\frac{\ell-1}{(r-1)(i-\ell)}\right) \nonumber\\
 = &~\frac{2}{r}\left(r-2-\frac{\ell-1}{i-\ell}\right)\left(h-\frac{r(r-2)}{4}+1\right).
\end{align}
By Inequality~(\ref{eq:maxl_i}) and Lemma~\ref{le:calculation-1}, we have
\begin{equation}\label{eq:maxl_i+}
\max\nolimits_{\ell}\{w_G(v_iv_j)\colon\, 1\leq j\leq i-1\}\geq h-((\ell-1)s+r-i)
\end{equation}
when (i) $2\leq \ell\leq s-2$ and $r-s+1\leq i\leq r-1$, or (ii) $\ell=s-1$ and $r-s+2\leq i\leq r-1$.

Let $F\colonequals G[\{v_1, \ldots, v_{r-1}\}]$.
\textbf{\boldmath  The remainder of the proof is devoted to showing that $F$ is $H$-friendly, } which together with Lemma~\ref{le:partite} implies that $G$ is $(r-1)$-partite.
Let $F^{+}$ be the multigraph obtained by adding a new vertex $v_r$ to $F$ and adding edges incident to $v_r$ of multiplicity at most $k$ such that $\sum_{i\in[r-1]}w_{F^+}(v_rv_i)\geq (r-2)k$ and $w_{F^+}(v_rv_i)\geq 1$ for all $i\in [r-1]$.
It suffices to show that $F^{+}$ contains a multicolored copy of $H$.
Note that $\min\{w_{F^+}(v_rv_i)\colon\, 1\leq i\leq r-1\}\geq 1.$
Since $r\geq 5$ and $k>\max\big\{\frac{r-1}{r-2}(h-1), 2\frac{r-1}{r}\big(h-\frac{r(r-2)}{4}+1\big)\big\}$, we have
\begin{align}\label{eq:min2_r}
 ~&~\min\nolimits_2\{w_{F^+}(v_rv_i)\colon\, 1\leq i\leq r-1\} \nonumber\\
 \geq &~\frac{1}{2}\left(\sum\nolimits_{i\in[r-1]}w_{F^+}(v_rv_i)-(r-3)k\right) ~\geq \frac{1}{2}((r-2)k-(r-3)k) \nonumber\\
 = &~\frac{k}{2} ~> \frac{r-1}{r}\left(h-\frac{r(r-2)}{4}+1\right),
\end{align}
\begin{align}\label{eq:min3_r}
 ~&~\min\nolimits_3\{w_{F^+}(v_rv_i)\colon\, 1\leq i\leq r-1\} \nonumber\\
 \geq &~\frac{1}{3}\left(\sum\nolimits_{i\in[r-1]}w_{F^+}(v_rv_i)-(r-4)k\right) ~\geq \frac{1}{3}((r-2)k-(r-4)k) \nonumber\\
 = &~\frac{2}{3}k ~> \frac{4}{3}\frac{r-1}{r}\left(h-\frac{r(r-2)}{4}+1\right) ~> h-\frac{r(r-2)}{4}+1,
\end{align}
\begin{align}\label{eq:max_r}
 ~&~\max\{w_{F^+}(v_rv_i)\colon\,1\leq i\leq r-1\} \nonumber\\
 \geq &~\frac{1}{r-1}\sum\nolimits_{i\in[r-1]}w_{F^+}(v_rv_i) ~\geq \frac{r-2}{r-1}k ~> \frac{r-2}{r-1}\cdot \frac{r-1}{r-2}(h-1) =~h-1,
\end{align}
and for $2\leq \ell \leq s-1$,
\begin{align}\label{eq:maxl_r}
 ~&~\max\nolimits_{\ell}\{w_{F^+}(v_rv_i)\colon\, 1\leq i\leq r-1\}\nonumber\\
 \geq &~\frac{1}{r-\ell}\left(\sum\nolimits_{i\in[r-1]}w_{F^+}(v_rv_i)-(\ell-1)k\right) ~\geq \frac{1}{r-\ell}((r-2)k-(\ell-1)k) \nonumber\\
 = &~\frac{r-\ell-1}{r-\ell}k ~> 2\frac{r-\ell-1}{r-\ell}\cdot \frac{r-1}{r}\left(h-\frac{r(r-2)}{4}+1\right).
\end{align}
Since the multiplicity of each edge is an integer, Inequality~(\ref{eq:max_r}) implies that
\begin{equation}\label{eq:max_r+}
\max\{w_{F^+}(v_rv_i)\colon\, 1\leq i\leq r-1\}\geq h.
\end{equation}
Moreover, recall that $h\geq {r\choose 2}$, so by Inequality~(\ref{eq:maxl_r}) and Lemma~\ref{le:calculation-2} we have
\begin{equation}\label{eq:maxl_r+}
\max\nolimits_{\ell}\{w_{F^+}(v_rv_i)\colon\, 1\leq i\leq r-1\}\geq h-(\ell-1)s
\end{equation}
when $2\leq \ell \leq s-1$.

For each $i\in [s]$, let $e_i$ be an edge in $\{v_{r-i+1}v_j\colon\, 1\leq j\leq r-i\}$ with $w_{F^{+}}(e_i)=\min\{w_{F^{+}}(v_{r-i+1}v_j)\colon\, 1\leq j\leq r-i\}$ (if there are more than one choice for $e_i$, we choose an arbitrary one). Then $w_{F^+}(e_1)\geq 1$ and
\begin{equation}\label{eq:e_i}
w_{F^+}(e_i)\geq (i-1)\frac{2}{r}\left(h-\frac{r(r-2)}{4}+1\right)
\end{equation}
for $2\leq i\leq s$ by Inequality~(\ref{eq:min_i}).
Let $T$ be the simple graph induced by the set of edges $\{e_1, \ldots, e_s\}$.
Note that for each $i\in [s-1]$, $e_i$ has at most one end-vertex in $T[\{e_{i+1}, \ldots, e_{s}\}]$.
Moreover, $T$ is a forest on at most $r$ vertices and $T$ has $s$ edges.
Let $e_{s+1}$ be an edge in $\{v_rv_j \colon\, 1\leq j\leq r-1\}$ with $w_{F^+}(e_{s+1})=\min_2\{w_{F^+}(v_rv_j)\colon\, 1\leq j\leq r-1\}$ (if there are more than one choice for $e_{s+1}$, we choose an arbitrary one, but it must be different from $e_1$).
Let $T^+$ be the simple graph induced by edges $\{e_1, \ldots, e_s, e_{s+1}\}$.
In other words, $T^{+}$ is the graph obtained from $T$ by adding the edge $e_{s+1}$ and its end-vertices (note that some of these two vertices may already be in $T$, and we do not add such vertices repeatedly).

\begin{claim}\label{cl:embedding}
There is an ordering $e_{s+2}, e_{s+3}, \ldots, e_{{r\choose 2}}$ of the edges in ${V(F^{+})\choose 2}\setminus E(T^{+})$ such that for each $j\in\{s+2, s+3, \ldots, {r\choose 2}\}$, we have $w_{F^{+}}(e_j)\geq h-\big({r\choose 2}-j\big)$.
\end{claim}

\begin{proof}
For each $r-s+1\leq i\leq r$, denote $\{v_iv_{j}\colon\, 1\leq j\leq i-1\}$ by $\{e_{i,1}, \ldots, e_{i,i-1}\}$ with $w_{F^+}(e_{i,1})\geq \cdots \geq w_{F^+}(e_{i,i-1})$.
Then for each $1\leq \ell \leq i-1$, we have $w_{F^+}(e_{i,\ell})=\max_{\ell}\{w(v_iv_j)\colon\, 1\leq j\leq i-1\}$.
We now define the edges $e_{s+2}, e_{s+3}, \ldots, e_{{r\choose 2}}$. Let $e_{{r\choose 2}}, e_{{r\choose 2}-1}, \ldots,$ $e_{{r\choose 2}-s(s-2)+1}$ be the edges
$$e_{r,1}, e_{r-1,1}, \ldots, e_{r-s+1,1}, ~ e_{r,2}, e_{r-1, 2}, \ldots, e_{r-s+1,2}, ~ \ldots, ~ e_{r,s-2}, e_{r-1,s-2}, \ldots, e_{r-s+1,s-2},$$
respectively. Let $e_{{r\choose 2}-s(s-2)}, \ldots, e_{{r\choose 2}-s(s-1)+2}$ be the edges $$e_{r,s-1}, e_{r-1,s-1}, \ldots, e_{r-s+2,s-1},$$ respectively.
Let $e_{{r\choose 2}-s(s-1)+1}, \ldots, e_{s+2}$ be an arbitrary ordering of the remaining edges in ${V(F^{+})\choose 2}\setminus E(T^{+})$.

By Inequalities~(\ref{eq:max_i}) and (\ref{eq:max_r+}), we have $w_{F^+}(e_j)\geq h-\left({r\choose 2}-j\right)$ for ${r\choose 2}-s+1\leq j\leq {r\choose 2}.$
For $s+2\leq j\leq {r\choose 2}-s(s-1)+1$, we have $h-\left({r\choose 2}-j\right)\leq h-s(s-1)+1\leq h-\frac{r(r-2)}{4}+1$.
By Inequalities~(\ref{eq:min2_i}), (\ref{eq:min3_r}) and since $F^+\{v_1,\ldots,v_{r-s}\}$ is a copy of $K_{r-s}^{(h-r(r-2)/4+1)}$, we have $w_{F^+}(e_j)\geq h-\left({r\choose 2}-j\right)$ for $s+2\leq j\leq {r\choose 2}-s(s-1)+1$.

Finally, we consider the case ${r\choose 2}-s(s-1)+2\leq j\leq {r\choose 2}-s$.
Note that for each $r-s+1 \leq i\leq r$ and $2\leq \ell \leq s-2$, we have $e_{i,\ell}=e_{{r\choose 2}-((\ell-1)s+(r+1-i))+1}=e_{{r\choose 2}-((\ell-1)s+r-i)},$ and for each $r-s+2 \leq i\leq r$, we have $e_{i,s-1}=e_{{r\choose 2}-(s(s-2)+(r+1-i))+1}=e_{{r\choose 2}-(s(s-2)+r-i)}.$
Hence, it suffices to show that $w_{F^+}(e_{i,\ell})\geq h-((\ell-1)s+r-i)$ when $r-s+1 \leq i\leq r$ and $2\leq \ell \leq s-2$, and $w_{F^+}(e_{i,s-1})\geq h-(s(s-2)+r-i)$ when $r-s+2 \leq i\leq r$.
This is indeed true by Inequalities~(\ref{eq:maxl_i+}) and (\ref{eq:maxl_r+}).
This completes the proof of Claim~\ref{cl:embedding}.
\end{proof}

Let $e_{s+2}, e_{s+3}, \ldots, e_{{r\choose 2}}$ be the ordering of edges in ${V(F^{+})\choose 2}\setminus E(T^{+})$ obtained by Claim~\ref{cl:embedding}.
Combining with Observation~\ref{ob:r vertex} (ii), for any $j\in\{s+2, s+3, \ldots, {r\choose 2}\}$ and any $E\subseteq {V(H)\choose 2}$ with $|E|=j$, we have
\begin{equation}\label{eq:embed}
w_{F^{+}}(e_j)\geq h-\left({r\choose 2}-j\right)\geq w_H(E).
\end{equation}
\textbf{\boldmath The remainder of the proof focuses on showing that there exist edges $f_1, \ldots, f_{s+1} \in {V(H)\choose 2}$ such that }
\begin{itemize}
\item[(i)] $H[\{f_1, \ldots, f_{s+1}\}]$ can be embedded into $T^+$ in which $f_i$ is embedded into $e_i$ for each $i\in [s+1]$, and
\item[(ii)] $w_{F^{+}}(e_i)\geq \sum_{j=1}^{i}w_H(f_j)$ for each $i\in [s+1]$.
\end{itemize}
This together with Inequality (\ref{eq:embed}) and Lemma~\ref{le:Hall} implies that $F^+$ contains a multicolored copy of $H$, and thus $F$ is $H$-friendly.
We divide the rest of the proof into two cases based on the value of $w_{F^+}(e_1)$.
Recall that $w_{F^+}(e_1)= \min\{w_{F^+}(v_rv_i)\colon\, 1\leq i\leq r-1\}\geq 1$.

\medskip
\noindent{\bf Case 1.} $w_{F^+}(e_1)> \frac{r-2}{r}\big(h-\frac{r(r-2)}{4}+1\big).$
\medskip

By the assumption of Case~1 and Inequalities~(\ref{eq:min2_r}) and (\ref{eq:e_i}), we have
$w_{F^+}(e_1)\geq \frac{r-2}{r}\big(h-\frac{r(r-2)}{4}+1\big),$ $w_{F^+}(e_{s+1})\geq \frac{r-1}{r}\big(h-\frac{r(r-2)}{4}+1\big),$ and
$w_{F^+}(e_j)\geq (j-1)\frac{2}{r}\big(h-\frac{r(r-2)}{4}+1\big)$ for $j\in \{2, \ldots, s\}$.
We use $e_1', e_2', \ldots, e_{{r\choose 2}}'$ to denote the edges in ${V(F^{+})\choose 2}$ with $e_i'\colonequals e_{i+1}$ for $i\in [s-2]$, $e_{s-1}'\colonequals e_{1}$, and $e_i'\colonequals e_i$ for $i\in \{s, s+1, \ldots, {r\choose 2}\}$. Then
\begin{equation}\label{eq:ej'}
w_{F^+}(e_i')\geq \begin{cases}
               i\frac{2}{r}\big(h-\frac{r(r-2)}{4}+1\big) & \mbox{if $i\in [s-2]$},\\
               \frac{r-2}{r}\big(h-\frac{r(r-2)}{4}+1\big) & \mbox{if $i=s-1$},\\
               \left(\left\lceil\frac{r}{2}\right\rceil-1\right)\frac{2}{r}\big(h-\frac{r(r-2)}{4}+1\big) & \mbox{if $i=s$},\\
               \frac{r-1}{r}\big(h-\frac{r(r-2)}{4}+1\big) & \mbox{if $i=s+1$}.
              \end{cases}
\end{equation}

Let $f'_1, f'_2, \ldots, f'_{{r\choose 2}}$ be an ordering of edges in ${V(H)\choose 2}$ with $w_H(f'_1)\geq w_H(f'_2)\geq \cdots \geq w_H\big(f'_{{r\choose 2}}\big)$.
For each $j\in [s+1]$, let $h_j\colonequals \sum_{\ell=s}^{s+j-1}w_{H}(f'_{\ell}).$

\begin{claim}\label{cl:hj}
For each $j\in [s+1]$, we have $h_j\leq \frac{j}{s+j-1}\big(h-\frac{r(r-2)}{4}+1\big).$
\end{claim}

\begin{proof} We shall show that $h_j\leq \frac{j}{s+j-1}\big(h-\big({r\choose 2}-(s+j-1)\big)\big)$, which implies the claim since ${r\choose 2}-(s+j-1)\geq {r\choose 2}-2s\geq {r\choose 2}-(r+1)\geq \frac{r(r-2)}{4}-1$.
Suppose for a contradiction that $h_j> \frac{j}{s+j-1}\big(h-\big({r\choose 2}-(s+j-1)\big)\big).$
Then $w_{H}(f'_s)\geq \frac{1}{j}h_j > \frac{1}{s+j-1}\big(h-\big({r\choose 2}-(s+j-1)\big)\big)$ and thus $w_H(f'_1)\geq \cdots \geq w_H(f'_{s-1})\geq w_{H}(f'_s)> \frac{1}{s+j-1}\big(h-\big({r\choose 2}-(s+j-1)\big)\big).$
Combining with Observation~\ref{ob:r vertex} (i), we have
\begin{align*}
e(H)= &~\sum\nolimits_{i=1}^{s-1}w_{H}(f'_i) + \sum\nolimits_{i=s}^{s+j-1}w_{H}(f'_i) + \sum\nolimits_{i=s+j}^{{r\choose 2}}w_{H}(f'_i) \\
> &~ (s-1)\frac{1}{s+j-1}\left(h-\left({r\choose 2}-(s+j-1)\right)\right) \\
~&~~+ \frac{j}{s+j-1}\left(h-\left({r\choose 2}-(s+j-1)\right)\right) + {r\choose 2}-(s+j-1) \\
= &~ h,
\end{align*}
a contradiction.
The result follows.
\end{proof}

\begin{claim}\label{cl:ej'}
For each $j\in [s+1]$, we have $w_{F^+}(e_j')\geq h_j.$
\end{claim}

\begin{proof} By Claim~\ref{cl:hj}, it suffices to show that $w_{F^+}(e_j')/\big(h-\frac{r(r-2)}{4}+1\big)\geq \frac{j}{s+j-1}$ for each $j\in [s+1].$
We shall use the lower bound on $w_{F^+}(e_j')$ given by Inequality~(\ref{eq:ej'}).
For $j\in [s-2]$, we have $j\frac{2}{r}-\frac{j}{s+j-1}\geq j\big(\frac{2}{r}-\frac{1}{r/2+1-1}\big)= 0.$
For $j=s-1$, we have $\frac{r-2}{r}-\frac{j}{s+j-1}= \frac{r-2}{r}-\frac{1}{2}> 0.$
For $j=s$, if $r\geq 5$ is odd, then $\left(\left\lceil\frac{r}{2}\right\rceil-1\right)\frac{2}{r}-\frac{j}{s+j-1}= \frac{r-1}{r}-\frac{r+1}{2r}>0$;
if $r\geq 6$ is even, then $\left(\left\lceil\frac{r}{2}\right\rceil-1\right)\frac{2}{r}-\frac{j}{s+j-1}= \frac{r-2}{r}-\frac{r}{2(r-1)}> 0$.
For $j=s+1$, we have $\frac{r-1}{r}-\frac{j}{s+j-1}= \frac{r-1}{r}-\frac{s+1}{2s}\geq \frac{r-1}{r}-\frac{1}{2}-\frac{1}{r}> 0.$
The result follows.
\end{proof}

Let $H'$ be the simple graph with $V(H')=V(H)$ and $E(H')=\{f'_1, f'_2, \ldots, f'_{s-1}\}$.
Recall that $T^+$ is an $(s+1)$-edge graph on at most $r$ vertices.
Let $T^{++}$ be the graph obtained from $T^+$ by adding $r-|V(T^+)|$ isolated vertices.
Since $e(T^{++})+e(H')=2s\leq r+1 \leq \frac{3}{2}(r-1)$, there is a packing of $T^{++}$ and $H'$ by Theorem~\ref{th:packing-sum}.
Hence, there exist $s+1$ edges $f'_{i_1}, \ldots, f'_{i_{s+1}}\in {V(H)\choose 2}$ and an embedding $\phi: V(H) \to V(F^+)$ such that
\begin{itemize}
\item[(i)] $f'_{i_1}, \ldots, f'_{i_{s+1}}\in \big\{f'_s, \ldots, f'_{{r\choose 2}}\big\}$, and
\item[(ii)] $\phi(f'_{i_j})=e_j'$ for each $j\in [s+1]$.
\end{itemize}
Moreover, by Claim~\ref{cl:ej'}, we have
$$w_{F^+}(e_j')\geq h_j=\sum\nolimits_{\ell=s}^{s+j-1}w_{H}(f'_{\ell}) \geq \sum\nolimits_{\ell=1}^{j}w_{H}(f'_{i_{\ell}})$$
for each $j\in [s+1]$.
For each $1\leq j\leq {r\choose 2}$, let $f''_j\colonequals \phi^{-1}(e_j')$.
Combining with Inequality (\ref{eq:embed}), we have $w_{F^{+}}(e_j')\geq \sum_{\ell=1}^{j}w_H(f''_{\ell})$ for each $1\leq j\leq {r\choose 2}$.
By Lemma~\ref{le:Hall}, there is a multicolored copy of $H$ in $F^+$.
Thus $F$ is $H$-friendly.
By Lemma~\ref{le:partite}, $G$ is $(r-1)$-partite.
This completes the proof of Case~1.

\medskip
\noindent{\bf Case 2.} $1\leq w_{F^+}(e_1)\leq \frac{r-2}{r}\big(h-\frac{r(r-2)}{4}+1\big).$
\medskip

In this case, we have
\begin{align}\label{eq:es+1}
 w_{F^+}(e_{s+1}) = &~\min\nolimits_2\{w(v_rv_i)\colon\, 1\leq i\leq r-1\} \nonumber\\
 \geq &~\sum\nolimits_{i\in[r-1]}w(v_rv_i)-(r-3)k-w_{F^+}(e_1) \nonumber\\
 \geq &~(r-2)k-(r-3)k-\frac{r-2}{r}\left(h-\frac{r(r-2)}{4}+1\right) \nonumber\\
 > &~2\frac{r-1}{r}\left(h-\frac{r(r-2)}{4}+1\right)-\frac{r-2}{r}\left(h-\frac{r(r-2)}{4}+1\right) \nonumber\\
 = &~h-\frac{r(r-2)}{4}+1 ~\geq h-\left({r\choose 2}-(s+1)\right).
\end{align}
Combining this with Observation~\ref{ob:r vertex} and Inequality~(\ref{eq:embed}), we have $w_{F^{+}}(e_j)\geq w_H(E)$ for any $j\in\{s+1, s+2, \ldots, {r\choose 2}\}$ and any $E\subseteq {V(H)\choose 2}$ with $|E|=j$.

Recall that $T$ is a forest with $s$ edges.
Let $T_1$ be the connected component of $T$ containing the edge $e_1$.
Let $m\colonequals |E(T_1)|-1$. Then $0\leq m\leq s-1$. Let $\{e_{i_1}, \ldots, e_{i_m}, e_1\}$ be the edge set of $T_1$.
Since $T_1$ is a tree, we may assume that for each $j\in [m]$, the edge $e_{i_j}$ has exactly one end-vertex in $T_1[\{e_{i_{j+1}}, \ldots, e_{i_m}, e_{1}\}].$
Let $e_1=v_rv_{q}$, where $q\in [r-1]$.
By the definition of $T$, if an edge is adjacent to $e_1$ in $T$, then this edge must be incident with $v_q$.
We may further assume that $\{e_{i_{m'}}, \ldots, e_{i_m}\}$ is the set of all edges adjacent to $e_{1}$ in $T_1$.
Then $\{e_{i_{m'}}, \ldots, e_{i_m}\}$ forms a star with center $v_q$.

\begin{claim}\label{cl:m'=1}
If $m'=1$ and $m=s-1$, then $G$ is $(r-1)$-partite.
\end{claim}

\begin{proof}
If $m'=1$ and $m=s-1$, then $T=T_1$ is a star with center $v_q$.
Let $V(H)=\{u_1, u_2, \ldots, u_r\}$ and $f_1\colonequals u_1u_2$ be a critical edge of $H$. So $w_H(f_1)=1$.
Without loss of generality, we may assume that $d_H(u_2)\leq d_H(u_1)$ and $w_H(u_2u_3)\leq \cdots \leq w_H(u_2u_r)$.
Then by Observation~\ref{ob:r vertex} (i), we have
$$\sum\nolimits_{i=3}^{r}w_H(u_2u_i)\leq \frac{1}{2}\left(h-\sum\nolimits_{e\in {V(H)\setminus \{u_1, u_2\}\choose 2}}w_H(e)-w_H(u_1u_2)\right)\leq \frac{1}{2}\left(h-{r-2 \choose 2}-1\right),$$
and moreover, for each $2\leq j\leq s$ we have
$$\sum\nolimits_{i=3}^{j+1}w_H(u_2u_{i})\leq \frac{j-1}{r-2}\sum\nolimits_{i=3}^{r}w_H(u_2u_i)\leq \frac{j-1}{r-2}\cdot \frac{1}{2}\left(h-{r-2 \choose 2}-1\right).$$
For each $2\leq j\leq s$, let $f_j\colonequals u_2u_{j+1}$, so
$$\sum\nolimits_{i=1}^{j}w_H(f_i)= 1+\sum\nolimits_{i=2}^{j}w_H(u_2u_{i+1})=1+\sum\nolimits_{i=3}^{j+1}w_H(u_2u_{i}) \leq 1+\frac{j-1}{r-2}\cdot \frac{1}{2}\left(h-{r-2 \choose 2}-1\right).$$
Recall that $r\geq 5$ and $h\geq {r\choose 2}$.
Then for $j\geq 2$, we have
\begin{align*}
~&~(j-1)\frac{2}{r}\left(h-\frac{r(r-2)}{4}+1\right)- 1-\frac{j-1}{r-2}\cdot \frac{1}{2}\left(h-{r-2 \choose 2}-1\right) \\
\geq &~(j-1)\frac{2}{r}\left({r\choose 2}-\frac{r(r-2)}{4}+1\right)- 1-\frac{j-1}{r-2}\cdot \frac{1}{2}\left({r\choose 2}-{r-2 \choose 2}-1\right) \\
= &~(j-1)\frac{2}{r}\left(\frac{r^2}{4}+1\right)-1-\frac{j-1}{2(r-2)}(2r-4) \\
= &~\frac{r}{2}(j-1)+\frac{2}{r}(j-1)-j ~> \frac{r}{2}(j-1)-j ~\geq \frac{5}{2}(j-1)-j ~= \frac{3}{2}j-\frac{5}{2} ~> 0,
\end{align*}
where the first inequality holds since $h\geq {r\choose 2}$ and $(j-1)\frac{2}{r}-\frac{j-1}{r-2}\cdot \frac{1}{2}=(j-1)\big(\frac{2}{r}-\frac{1}{2(r-2)}\big)>0.$
Thus $\sum_{i=1}^{j}w_H(f_i)< (j-1)\frac{2}{r}\big(h-\frac{r(r-2)}{4}+1\big)$ for $2\leq j\leq s$.
This together with Inequality~(\ref{eq:e_i}) implies that $w_{F^+}(e_j)\geq \sum_{i=1}^{j}w_H(f_i)$ for each $j\in [s]$.
Let $\phi: V(H) \to V(F^+)$ be an embedding with $\phi(f_j)=e_j$ for each $j\in [s]$.
Combining with Observation~\ref{ob:r vertex}, Inequalities~(\ref{eq:embed}), (\ref{eq:es+1}) and Lemma~\ref{le:Hall}, we can deduce that $F^+$ contains a multicolored copy of $H$.
Thus $F$ is $H$-friendly.
By Lemma~\ref{le:partite}, $G$ is $(r-1)$-partite.
This completes the proof of Claim~\ref{cl:m'=1}.
\end{proof}

By Claim~\ref{cl:m'=1}, we may assume that either $m'\neq 1$ or $m\neq s-1$ in the following arguments.

\begin{claim}\label{cl:T1}
$H$ contains a subgraph $H_1$ such that
\begin{itemize}
\item[{\rm (i)}] the underlying graph  of $H_1$ is $T_1$,
\item[{\rm (ii)}] the edge corresponding to $e_1$ is a critical edge of $H$, and
\item[{\rm (iii)}] each edge of $H_1$ has multiplicity at most $\frac{2}{r}\big(h-\frac{r(r-2)}{4}+1\big)-1$ in $H$.
\end{itemize}
\end{claim}

\begin{proof}
If $m=0$, then $T_1$ consists of the single edge $e_1$.
In this case, the result holds clearly since $H$ contains a critical edge.
In the following, we may assume that $m\geq 1$.
Let $f_1\colonequals u_1u_2$ be a critical edge of $H$.
Without loss of generality, we may assume that $d_H(u_2)\leq d_H(u_1)$.
Then $u_2$ has a neighbor $u_3\in V(H)\setminus \{u_1,u_2\}$ with $w_H(u_2u_3)\leq \frac{1}{r-2}\cdot \frac{1}{2}\left(h-{r-2 \choose 2}-1\right)\leq \frac{2}{r}\big(h-\frac{r(r-2)}{4}+1\big)-1.$
Let $f_{i_m}\colonequals u_2u_3.$
Assume that we have found $f_{i_{\ell}}, \ldots, f_{i_m}, f_1\in {V(H)\choose 2}$ for some $2\leq \ell \leq m$ such that for each $j\in \{\ell, \ldots, m\}$,
\begin{itemize}
\item[{\rm (i)}] the edge $f_{i_j}$ has multiplicity at most $\frac{2}{r}\big(h-\frac{r(r-2)}{4}+1\big)-1$, and
\item[{\rm (ii)}] the underlying graph of $H_{(j)}\colonequals H[\{f_{i_j}, \ldots, f_{i_m}, f_1\}]$ is isomorphic to $T_1[\{e_{i_j}, \ldots, e_{i_m}, e_1\}]$.
\end{itemize}
We shall find an edge $f_{i_{\ell-1}}\in {V(H)\choose 2}$ with $w_H(f_{i_{\ell-1}})\leq \frac{2}{r}\big(h-\frac{r(r-2)}{4}+1\big)-1$ such that $H[\{f_{i_{\ell-1}}, \ldots, f_{i_m}, f_1\}]$ is isomorphic to $T_1[\{e_{i_{\ell-1}}, \ldots, e_{i_m}, e_1\}]$.

Recall that $e_{i_{\ell-1}}$ has exactly one end-vertex in $T_1[\{e_{i_{\ell}}, \ldots, e_{i_m}, e_1\}]$, say vertex $w$.
Let $w'$ be the vertex in $H_{(\ell)}$ corresponding to $w$.
Note that $|V(H_{(\ell)})|=|E(H_{(\ell)})|+1= m-\ell+3$.
Let $w''\in V(H)\setminus  V(H_{(\ell)})$ with $w_H(w'w'')=\min_{x\in V(H)\setminus  V(H_{(\ell)})}w_H(w'x).$
By Observation~\ref{ob:r vertex} (i) and (iii), we have
$$w_H(w'w'')\leq \frac{d_H(w')-d_{H_{(\ell)}}(w')}{r-|V(H_{(\ell)})|} \leq \frac{h-{r-1 \choose 2}-(|V(H_{(\ell)})|-1)}{r-|V(H_{(\ell)})|} = 1+\frac{h-{r \choose 2}}{r-(m-\ell+3)}.$$
If $r\geq 6$ is even, then
$$w_H(w'w'')\leq 1+\frac{h-{r \choose 2}}{r-((s-1)-2+3)}= 1+\frac{h-{r \choose 2}}{r-r/2}\leq \frac{2}{r}\left(h-\frac{r(r-2)}{4}+1\right)-1;$$
if $r\geq 5$ is odd and $m-\ell \leq s-4$, then
$$w_H(w'w'')\leq 1+\frac{h-{r \choose 2}}{r-s+1}= 1+\frac{h-{r \choose 2}}{r-(r+1)/2+1}\leq \frac{2}{r}\left(h-\frac{r(r-2)}{4}+1\right)-1.$$
In both cases, we can choose $f_{i_{\ell-1}}$ to be $w'w''$.

Now we consider the remaining case that $r\geq 5$ is odd and $m-\ell \geq s-3$.
Since $m\leq s-1$ and $\ell \geq 2$, we have $m-\ell\leq s-3$.
Thus $m-\ell=s-3$, $m=s-1$ and $\ell=2$ now.
In the following, we shall find a desired edge $f_{i_1}$.
Recall that we have either $m'\neq 1$ or $m\neq s-1$.
Thus $m'\neq 1$.
Then $e_{i_1}$ is not adjacent to $e_1$, so $w$ is not an end-vertex of $e_1$.
Thus $w'\notin\{u_1, u_2\}$.
If $w_H(w'w'')\leq \frac{2}{r}\big(h-\frac{r(r-2)}{4}+1\big)-1$, then we can choose $f_{i_1}$ to be $w'w''$, and we are done.
If $w_H(w'w'')>\frac{2}{r}\big(h-\frac{r(r-2)}{4}+1\big)-1$, then since $w_H(w'w'')=\min_{x\in V(H)\setminus V(H_{(\ell)})}w_H(w'x)$ and $|V(H)\setminus V(H_{(2)})|=r-(m-\ell+3)= r-s= \frac{r-1}{2}$, we have
\begin{align*}
 \sum_{x\in V(H)\setminus V(H_{(2)})}w_H(w'x) > &~ |V(H)\setminus V(H_{(2)})|\left(\frac{2}{r}\left(h-\frac{r(r-2)}{4}+1\right)-1\right) \\
 = &~\frac{r-1}{r}\left(h-\frac{r(r-2)}{4}+1\right)-\frac{r-1}{2}.
\end{align*}
By Observation~\ref{ob:r vertex} (i), for any $e\in {V(H)\choose 2}\setminus \{w'x\colon\, x\in V(H)\setminus  V(H_{(2)})\}$, we have
\begin{align*}
 w_H(e)\leq &~h-\sum_{x\in V(H)\setminus V(H_{(2)})}w_H(w'x)-\left({r \choose 2}-\left(|V(H)\setminus V(H_{(2)})|\right)-1\right) \\
 < &~h-\left(\frac{r-1}{r}\left(h-\frac{r(r-2)}{4}+1\right)-\frac{r-1}{2}\right)-\left({r \choose 2}-\frac{r-1}{2}-1\right) \\
 = &~\frac{1}{r}h+\frac{(r-1)(r-2)}{4}-\frac{r-1}{r} +(r-1)-{r\choose 2}+1 \\
 = &~\frac{1}{r}h-\frac{(r-1)(r-2)}{4}+\frac{1}{r} ~< \frac{2}{r}\left(h-\frac{r(r-2)}{4}+1\right)-1.
\end{align*}
Hence, we can choose two vertices $w^{\ast}, w^{\ast\ast}\in V(H)\setminus V(H_{(2)})$ arbitrarily so that if we replace $w'$ by $w^{\ast}$ in $H_{(2)}$ and let $f_{i_1}\colonequals w^{\ast}w^{\ast\ast}$, then we get a desired $H_1$.
This completes the proof of Claim~\ref{cl:T1}.
\end{proof}

Let $t_1\colonequals |E(T_1)|$ and $T_2\colonequals T\setminus T_1$.
Then $T_2$ is a forest on $s-t_1$ edges.
The following Claims~\ref{cl:r=5} and \ref{cl:r6} ensure the existences of the desired edges $f_1, \ldots, f_{s}$.
We first show that we can find the desired edges $f_1, \ldots, f_{s}$ in the case $r=5$.
Note that $s=3$ when $r=5$.

\begin{claim}\label{cl:r=5}
If $r=5$, then we can find edges $f_1, f_2, f_{3} \in {V(H)\choose 2}$ such that
\begin{itemize}
\item[{\rm (i)}] $H[\{f_1, f_2, f_{3}\}]$ can be embedded into $T$ in which $f_i$ is embedded into $e_i$ for each $i\in [3]$, and
\item[{\rm (ii)}] $w_{F^{+}}(e_i)\geq \sum_{j=1}^{i}w_H(f_j)$ for each $i\in [3]$.
\end{itemize}
\end{claim}

\begin{proof}
Since $|V(T)|\leq r=5$, $T$ has at most two components.
Let $V(H)=\{u_1, u_2, \ldots, u_5\}$ with $w_H(u_1u_2)=1$.
If $t_1=3=s$, then we are done by Claim~\ref{cl:T1} and Inequality~(\ref{eq:e_i}).
So we may assume that $t_1\leq 2$.

If $t_1=1$, then $T_1=K_2$ and $T_2=K_{1,2}$.
Note that $$\min\left\{w_H(a)+w_H(b)\colon\, a, b\in \{u_3u_4, u_3u_5, u_4u_5\}, a\neq b\right\}\leq \frac{2}{3}(h-7),$$
since otherwise
\begin{align*}
e(H)\geq &~\frac{1}{2}\left(w_H(u_3u_4)+w_H(u_3u_5)+w_H(u_3u_4)+w_H(u_4u_5)+w_H(u_3u_5)+w_H(u_4u_5)\right)+7 \\
> &~\frac{3}{2}\cdot \frac{2}{3}(h-7)+7 ~= h,
\end{align*}
a contradiction.
Let $a, b\in \{u_3u_4, u_3u_5, u_4u_5\}$ with $a\neq b$, $w_H(a)+w_H(b)\leq \frac{2}{3}(h-7) \leq \frac{4}{r}\big(h-\frac{r(r-2)}{4}+1\big)-1\leq w_{F^+}(e_3)-1$ and $w_H(a)\leq w_H(b)$.
Then $w_H(a)\leq \frac{1}{3}(h-7) \leq \frac{2}{r}\big(h-\frac{r(r-2)}{4}+1\big)-1\leq w_{F^+}(e_2)-1$.
In this case, $u_1u_2$ forms a desired $T_1$, and $a$ and $b$ form a desired $T_2$.
Hence, we can assign $(f_1, f_2, f_3)$ to be $(u_1u_2, a, b)$.

If $t_1=2$, then $T_1=K_{1,2}$ and $T_2=K_2$.
Note that $$\min\left\{w_H(u_1u_3)+w_H(u_4u_5), w_H(u_1u_4)+w_H(u_3u_5), w_H(u_1u_5)+w_H(u_3u_4)\right\}\leq \frac{1}{3}(h-4).$$
Without loss of generality, we may assume that $w_H(u_1u_3)+w_H(u_4u_5)\leq \frac{1}{3}(h-4)\leq \frac{2}{r}\big(h-\frac{r(r-2)}{4}+1\big)$.
Then $w_H(u_1u_3)\leq \frac{2}{r}\big(h-\frac{r(r-2)}{4}+1\big)-1$ and $w_H(u_4u_5)\leq \frac{2}{r}\big(h-\frac{r(r-2)}{4}+1\big)-1$.
In this case, $u_1u_2$ and $u_1u_3$ form a desired $T_1$, and $u_4u_5$ forms a desired $T_2$.
Hence, we can assign $(f_1, f_2, f_3)$ to be one of $(u_1u_2, u_1u_3, u_4u_5)$ and $(u_1u_2, u_4u_5, u_1u_3)$.
This completes the proof of Claim~\ref{cl:r=5}.
\end{proof}

We next show that we can find the desired edges $f_1, \ldots, f_{s}$ in the case $r\geq 6$.
Note that the statement of Claim~\ref{cl:r6} below is similar to that of Claim~\ref{cl:T1}.
For the convenience of readers, we now point out their differences and connections.
Claim~\ref{cl:T1} aims to find a subgraph $H_1$ of $H$ so that we can embed it into $T_1$,
while Claim~\ref{cl:r6} aims to find a subgraph $H_T$ of $H$ so that we can embed it into $T$.
Here $T_1$ is a connected component of $T$, and $H_1$ is a connected component of $H_T$.
In the proof of Claim~\ref{cl:r6}, we shall find a required $H_T$ on the basis of $H_1$ given by Claim~\ref{cl:T1}.

\begin{claim}\label{cl:r6}
If $r\geq 6$, then $H$ contains a subgraph $H_T$ such that
\begin{itemize}
\item[{\rm (i)}] the underlying graph of $H_T$ is $T$,
\item[{\rm (ii)}] the edge corresponding to $e_1$ is a critical edge of $H$, and
\item[{\rm (iii)}] each edge of $H_T$ has multiplicity at most $\frac{2}{r}\big(h-\frac{r(r-2)}{4}+1\big)-1$ in $H$.
\end{itemize}
\end{claim}

\begin{proof}
Recall that $t_1= |E(T_1)|$ and $T_2= T\setminus T_1$ is a forest with $s-t_1$ edges and at most $r-(t_1+1)$ vertices.
If $t_1=s$ (i.e., $T=T_1$), then the result follows from Claim~\ref{cl:T1}.
Hence, we may assume that $t_1\leq s-1 \leq \frac{r-1}{2}$.

Let $E_1\colonequals \big\{e\in {V(H)\choose 2}\colon\, w_H(e)>\frac{2}{r}\big(h-\frac{r(r-2)}{4}+1\big)-1\big\}.$
Then $|E_1|< \frac{r}{2}$; otherwise $$e(H)> \frac{r}{2}\left(\frac{2}{r}\left(h-\frac{r(r-2)}{4}+1\right)-1\right)+{r\choose 2}-\frac{r}{2} = h-\frac{r(r-2)}{4}+{r-1 \choose 2} > h,$$ a contradiction.
Since $|E_1|$ is an integer, we further have $|E_1|\leq s-1$.
Let $H_1$ be the subgraph of $H$ obtained by Claim~\ref{cl:T1}.
Let $H^{\ast}$ be the subgraph of $H$ with $V(H^{\ast})=V(H)\setminus V(H_1)$ and $E(H^{\ast})=\big\{e\in E(H[V(H^{\ast})])\colon\, w_H(e)\leq \frac{2}{r}\big(h-\frac{r(r-2)}{4}+1\big)-1\big\}.$
Let $H^{\ast\ast}$ be the graph with $V(H^{\ast\ast})=V(H^{\ast})=V(H)\setminus V(H_1)$ and $E(H^{\ast\ast})=E_1\cap {V(H^{\ast\ast})\choose 2}$.
Let $T_2'$ be the graph obtained from $T_2$ by adding $|V(H^{\ast\ast})|-|V(T_2)|$ isolated vertices.
In order to prove the claim, we shall show that $H^{\ast}$ contains a subgraph $H_2$ whose underlying graph is $T_2$.
It suffices to show that there is a packing of $T_2'$ and $H^{\ast\ast}$.

Note that $|V(T_2')|=|V(H^{\ast\ast})|=r-(t_1+1)$, $e(T_2')=e(T_2)=s-t_1$ and $e(H^{\ast\ast})\leq |E_1|\leq s-1$.
Then
\begin{align*}
~&~\frac{3}{2}(|V(H^{\ast\ast})|-1)-e(T_2')-e(H^{\ast\ast}) \\
\geq &~\frac{3}{2}(r-t_1-2)-(s-t_1)-(s-1) ~= \frac{3r}{2}-2s-\frac{t_1}{2}-2 \\
\geq &~\frac{3r}{2}-2s-\frac{s-1}{2}-2 ~= \frac{1}{2}(3r-5s-3).
\end{align*}
If $r\geq 6$ is even, then $3r-5s-3=3r-5\frac{r}{2}-3= \frac{r}{2}-3\geq 0$; if $r\geq 11$ is odd, then $3r-5s-3=3r-5\frac{r+1}{2}-3= \frac{r-11}{2} \geq 0$.
Combining with Theorem~\ref{th:packing-sum}, there is a packing of $T_2'$ and $H^{\ast\ast}$, unless $r\in \{7,9\}$.

If $r=9$, then $s=5$, $1\leq t_1\leq 4$, $e(T_2')=5-t_1$ and $e(H^{\ast\ast})\leq |E_1|\leq 4$.
Thus $e(T_2')e(H^{\ast\ast})\leq 4(5-t_1)<{8-t_1\choose 2} ={|V(H^{\ast\ast})|\choose 2}$.
By Theorem~\ref{th:packing-prod}, there is a packing of $T_2'$ and $H^{\ast\ast}$.

If $r=7$, then $s=4$, $1\leq t_1\leq 3$ and $e(H^{\ast\ast})\leq |E_1|\leq 3$.
Firstly, if $t_1=1$, then $|V(H^{\ast\ast})|=7-2=5$ and $e(T_2')=3$.
Since $e(T_2')e(H^{\ast\ast})\leq 9<{5\choose 2}={|V(H^{\ast\ast})|\choose 2}$, there is a packing of $T_2'$ and $H^{\ast\ast}$ by Theorem~\ref{th:packing-prod}.
Secondly, if $t_1=2$, then $|V(H^{\ast\ast})|=7-3=4$ and $e(T_2)=2$.
Let $V(H_1)=\{u_1, u_2, u_3\}$, $V(H)\setminus V(H_1)=\{u_4, u_5, u_6, u_7\}$, and $u_1u_2$ be a critical edge.
If $\big|E_1\cap {V(H)\setminus V(H_1) \choose 2}\big|\leq 2$, then one can easily find a desired $T_2$ within $V(H)\setminus V(H_1)$.
If $\big|E_1\cap {V(H)\setminus V(H_1) \choose 2}\big|=3$, then $E_1\cap \{u_iu_j\colon\, i\in [3], j\in \{4, 5, 6, 7\}\}= \emptyset$.
Since ${|V(H)\setminus V(H_1)|\choose 2}=6>3$, we may assume that $u_4u_5\notin E_1$ without loss of generality.
Then $u_1u_2$ and $u_1u_6$ form a desired $T_1$, and $u_4u_5$ and one edge in $\{u_3u_4, u_3u_7\}$ form a desired $T_2$.
Finally, if $t_1=3$, then $T_1\in \{K_{1,3}, P_4\}$, $T_2=K_2$ and $|V(H^{\ast\ast})|=7-4=3$.
Let $V(H_1)=\{u_1, u_2, u_3, u_4\}$, $V(H)\setminus V(H_1)=\{u_5, u_6, u_7\}$, and $u_1u_2$ be a critical edge.
If ${V(H)\setminus V(H_1) \choose 2}\setminus E_1\neq \emptyset$, then we can find a desired $T_2$ within $V(H)\setminus V(H_1)$.
If ${V(H)\setminus V(H_1) \choose 2}\setminus E_1= \emptyset$, then $E_1\cap \big({V(H)\choose 2}\setminus {V(H)\setminus V(H_1) \choose 2}\big)=\emptyset$.
Then one can easily find a desired $T$ consisting of edges in ${V(H)\choose 2}\setminus {V(H)\setminus V(H_1) \choose 2}$.
This completes the proof of Claim~\ref{cl:r6}.
\end{proof}

By Claims~\ref{cl:r=5}, \ref{cl:r6} and Inequality~(\ref{eq:e_i}), there exist edges $f_1, \ldots, f_{s} \in {V(H)\choose 2}$ and an embedding $\phi: V(H) \to V(F^+)$ such that for each $i\in [s]$,
\begin{itemize}
\item[(i)] $\phi(f_i)=e_i$, and
\item[(ii)] $w_{F^{+}}(e_i)\geq \sum_{j=1}^{i}w_H(f_j)$.
\end{itemize}
Recall that $e_{s+1}$ is an edge in $\{v_rv_j \colon\, 1\leq j\leq r-1\}$ with $w_{F^+}(e_{s+1})=\min_2\{w_{F^+}(v_rv_j)\colon\, 1\leq j\leq r-1\}$, and $e_{s+2}, e_{s+3}, \ldots, e_{{r\choose 2}}$ is the ordering of edges in ${V(F^{+})\choose 2}\setminus E(T^{+})$ obtained by Claim~\ref{cl:embedding}.
For each $s+1\leq i\leq {r\choose 2}$, let $f_i\colonequals \phi^{-1}(e_i)$.
By Inequalities~(\ref{eq:embed}) and (\ref{eq:es+1}), we have $w_{F^{+}}(e_i)\geq \sum_{j=1}^{i}w_H(f_j)$ for each $s+1\leq i\leq {r\choose 2}$.
By Lemma~\ref{le:Hall}, there is a multicolored copy of $H$ in $F^+$.
Thus $F$ is $H$-friendly.
By Lemma~\ref{le:partite}, $G$ is $(r-1)$-partite.
This completes the proof of Lemma~\ref{le:r_vertex-degree}.
\hfill$\square$
\vspace{0.2cm}


\section{Step II: completing the proof via stability}
\label{sec:any_vertex}

In this section, we finish the proof of Theorem~\ref{th:main}.
First, we prove a stability result for $r$-vertex $r$-color-critical multigraphs (i.e., Lemma~\ref{le:r_vertex-stability}) using the ideas in \cite{CKLLS,RoSc}.
Second, we extend the stability result to $r$-color-critical graphs with any number of vertices (i.e., Lemma~\ref{le:any_vertex-stability}) using the multicolor version of Szemer\'{e}di's Regularity Lemma and the Embedding Lemma.
Last, we prove Theorem~\ref{th:main+}, which implies Theorem~\ref{th:main} immediately.

We will use the following observation which follows from a straightforward calculation.

\begin{observation}{\normalfont (\cite[Proposition~6.2]{CKLLS})}\label{ob:G-B}
Suppose $0<\frac{1}{n}\ll \delta <1$. Let $G$ be an $n$-vertex multigraph with $d{n\choose 2}$ edges.
If $B\subseteq V(G)$ is a vertex set with $|B|=\frac{1}{2}\delta n$ and every $v\in B$ satisfies $d_G(v)<(1-\delta)dn$, then $G-B$ has at least $\big(1+\frac{1}{2}\delta^2\big)d{|V(G-B)|\choose 2}$ edges.
\end{observation}

For two multigraphs $G$ and $H$ of the same order, the {\it symmetric difference} of $G$ and $H$ is defined by $$|G \triangle H|\colonequals \min\limits_{\substack{H'\cong H, \\ V(H')=V(G)}}~\sum_{e\in {V(G)\choose 2}}|w_G(e)-w_{H'}(e)|.$$

\begin{lemma}\label{le:r_vertex-stability}
Let $r\geq 5$, $k> \max\big\{\frac{r-1}{r-2}(h-1), 2\frac{r-1}{r}\big(h-\frac{r(r-2)}{4}+1\big)\big\}$ and $0<\frac{1}{\sqrt{n}}\leq \eta \ll \eta_1, \frac{1}{k}<1$.
Let $H$ be an $r$-vertex $r$-color-critical multigraph with $h$ edges.
Let $G$ be an $n$-vertex simply $k$-colored multicolored-$H$-free multigraph with $e(G)\geq k\cdot t_{r-1}(n)-\eta n^2$.
Then $|G\triangle (k\cdot T_{r-1}(n))|\leq \eta_1 n^2$.
\end{lemma}

\begin{proof}
Let $\delta$ be such that $0<\eta \ll \delta \ll \eta_1, \frac{1}{k}<1$.
We first show that $G$ contains a subgraph of order at least $(1-\delta^{1/2})n$ with minimum degree at least $(1-\delta^{1/2})k\delta(T_{r-1}(n))$.
Let $d>0$ be the number such that $e(G)=d{n\choose 2}$.
Since $e(G)\geq k\cdot t_{r-1}(n)-\eta n^2\geq k\frac{r-2}{r-1}{n\choose 2}-\eta n^2$, we have $d\geq k\frac{r-2}{r-1}-3\eta$.
Let $L\colonequals \{v\in V(G)\colon\, d(v)< (1-\delta)dn\}$.

\begin{claim}\label{cl:L} $|L|<\delta\frac{n}{2}.$
\end{claim}

\begin{proof} Suppose for a contradiction that $|L|\geq \delta\frac{n}{2}$.
We choose $B\subseteq L$ with $|B|=\delta \frac{n}{2}$.
By Observation~\ref{ob:G-B} and since $\frac{1}{\sqrt{n}}\leq \eta \ll \delta$, we have
\begin{align*}
 e(G-B)\geq &~\left(1+\frac{1}{2}\delta^2\right)d{|V(G-B)|\choose 2} ~\geq \left(1+\frac{1}{2}\delta^2\right)\left(k\frac{r-2}{r-1}-3\eta\right){|V(G-B)|\choose 2} \\
  > &~\left(k\frac{r-2}{r-1}-3\eta + \frac{1}{2}\delta^2\right){|V(G-B)|\choose 2} ~> \left(k\frac{r-2}{r-1}+ \frac{1}{3}\delta^2\right){|V(G-B)|\choose 2} \\
  > &~k\cdot t_{r-1}(|V(G-B)|).
\end{align*}
On the other hand, since $G$ (and thus $G-B$) is multicolored-$H$-free, we have $e(G-B)\leq k\cdot t_{r-1}(|V(G-B)|)$ by Theorem~\ref{th:r_vertex}.
This contradiction completes the proof.
\end{proof}

Let $M\colonequals G-L$.
By Claim~\ref{cl:L}, we have $|L|<\delta\frac{n}{2}.$
Combining with $r\geq 5$, $k> \frac{r-1}{r-2}(h-1)$, $0<\eta \ll \delta \ll \frac{1}{k}<1$ and Inequalities~(\ref{eq:t_r-1}) and (\ref{eq:degree_T_r-1}), we have
\begin{align}\label{eq:dM}
 \delta(M)\geq &~(1-\delta)dn-k|L| ~> (1-\delta)\left(k\frac{r-2}{r-1}-3\eta\right)n-k\delta\frac{n}{2} \nonumber\\
 = &~\left(k\frac{r-2}{r-1}-3\eta-\left(k\left(\frac{r-2}{r-1}+\frac{1}{2}\right)-3\eta\right)\delta\right)n ~\geq \left(1-\delta^{1/2}\right)k\frac{r-2}{r-1}n \nonumber\\
 \geq &~\left(1-\delta^{1/2}\right)k\delta(T_{r-1}(n)) ~\geq \left(1-\delta^{1/2}\right)k\delta(T_{r-1}(|V(M)|)),
\end{align}
and
\begin{align}\label{eq:eM}
  e(M)\geq &~\frac{1}{2}\delta(M)|V(M)| ~> \frac{1}{2}\left(1-\delta^{1/2}\right)k\delta(T_{r-1}(|V(M)|))|V(M)| \nonumber\\
 \geq &~\frac{1}{2}\left(1-\delta^{1/2}\right)k\frac{r-2}{r-1}(|V(M)|-1)|V(M)| ~= \left(1-\delta^{1/2}\right)k\frac{r-2}{r-1}{|V(M)|\choose 2} \nonumber\\
 \geq &~k\frac{r-2}{r-1}\cdot \frac{|V(M)|^2}{2}-\delta^{2/5}n^2 ~\geq k\cdot t_{r-1}(|V(M)|)-\delta^{2/5}n^2.
\end{align}
By Inequality~(\ref{eq:dM}) and Lemma~\ref{le:r_vertex-degree}, we can deduce that $M$ is $(r-1)$-partite.
Let $V_1, \ldots, V_{r-1}$ be the partite sets of $M$.
We claim that for each $i\in [r-1]$, we have $\big||V_i|-\frac{|V(M)|}{r-1}\big|\leq 2\delta^{1/5}n.$
Otherwise, there exist some $i\neq j$ such that $|V_i|-|V_j|>2\delta^{1/5}n$, which implies that $e(M)\leq k\cdot t_{r-1}(|V(M)|)-\delta^{2/5}n^2$, contradicting Inequality~(\ref{eq:eM}).
Then, by deleting at most $k(r-1)(2\delta^{1/5}n)n<\delta^{1/6}n^2$ edges of $M$, we obtain an $(r-1)$-partite multigraph $M'$ with class sizes equal to that of $T_{r-1}(|V(M)|)$, and $$e(M')\geq e(M)-\delta^{1/6}n^2>k\cdot t_{r-1}(|V(M)|)-\delta^{2/5}n^2-\delta^{1/6}n^2 > k\cdot t_{r-1}(|V(M)|)- 2\delta^{1/6}n^2.$$
Then $$|G\triangle (k\cdot T_{r-1}(n))|\leq |M\triangle (k\cdot T_{r-1}(|V(M)|))|+k|L|n < \delta^{1/6}n^2 + 2\delta^{1/6}n^2 + k\delta\frac{n^2}{2}< \eta_1 n^2.$$
The proof is complete.
\end{proof}

In order to extend the stability result from $r$-vertex $r$-color-critical multigraphs to $r$-color-critical graphs with any number of vertices, we shall use the multicolor version of Szemer\'{e}di's Regularity Lemma \cite{Sze} and the Embedding Lemma.
Similar results for edge-colored graphs were also used in \cite{HoLO17SIAM-DM}.
Let $G$ be a simply $k$-colored multigraph.
The {\it edge density} of $G$ is defined as $d(G)\colonequals \frac{e(G)}{|V(G)|^2}$.
For two disjoint nonempty vertex sets $X$, $Y$ and a color $\rho$, the {\it $\rho$-density} of $(X,Y)$ is defined to be $d_{\rho}(X,Y)\colonequals \frac{e_{\rho}(X,Y)}{|X||Y|},$ where $e_{\rho}(X,Y)$ is the number of edges between $X$ and $Y$ with color $\rho$.
For $\varepsilon >0$, the pair $(X,Y)$ is {\it $\varepsilon$-regular} if for every $X'\subseteq X$ and $Y'\subseteq Y$ with $|X'|\geq \varepsilon |X|$ and $|Y'|\geq \varepsilon |Y|$, we have $|d_{\rho}(X,Y)-d_{\rho}(X',Y')|\leq \varepsilon$ for every color $\rho$.
A partition $\mathcal{P}=(V_1, \ldots, V_m)$ of $V(G)$ is an {\it $\varepsilon$-regular partition} of $G$ if
\begin{itemize}
\item $\left||V_i|-|V_j|\right|\leq 1$ for all $1\leq i<j\leq m$, and
\item $(V_i, V_j)$ is $\varepsilon$-regular for all but at most $\varepsilon m^2$ pairs $(i,j)$.
\end{itemize}
For $\varepsilon, \gamma >0$ and a color $\rho$, the pair $(X,Y)$ is {\it $(\varepsilon, \gamma; \rho)$-lower-regular} if for every $X'\subseteq X$ and $Y'\subseteq Y$ with $|X'|\geq \varepsilon |X|$ and $|Y'|\geq \varepsilon |Y|$, we have $d_{\rho}(X',Y')\geq \gamma$.
Given $\varepsilon, \gamma >0$ and an $\varepsilon$-regular partition $\mathcal{P}=(V_1, \ldots, V_m)$ of $G$, we define the {\it $(\varepsilon, \gamma, \mathcal{P})$-reduced multigraph} $R$ as follows:
$R$ is a simply $k$-colored multigraph with colors $\{R_1, \ldots, R_k\}$ and vertex set $[m]$, and for each $ij\in {[m]\choose 2}$ and $\rho\in [k]$, we have $ij \in E(R_{\rho})$ if and only if $(V_i,V_j)$ is $(\varepsilon, \gamma; \rho)$-lower-regular.

\begin{lemma}[Multicolor Regularity Lemma \cite{CKLLS,KoSi,Sze}]\label{le:regulartity}
For any $\varepsilon >0$ and integers $k, M_0\geq 1$, there exist $n'$ and $M$ such that every simply $k$-colored multigraph $G$ on $n\geq n'$ vertices admits an $\varepsilon$-regular partition $\mathcal{P}=(V_1, \ldots, V_m)$ with $M_0\leq m\leq M$.
Moreover, for $\gamma >0$, the density of the $(\varepsilon, \gamma, \mathcal{P})$-reduced multigraph $R$ satisfies $d(R)\geq d(G)-2(\varepsilon + \gamma)$.
\end{lemma}

Given a simply $k$-colored multigraph $G$ with an $\varepsilon$-regular partition $\mathcal{P}=(V_1, \ldots, V_m)$ and $(\varepsilon, \gamma, \mathcal{P})$-reduced multigraph $R$, we define a simply $k$-colored multigraph $G^{\mathcal{P}}=G^{\mathcal{P}}(\varepsilon, \gamma)$ as follows.
Let $G^{\mathcal{P}}$ be a simply $k$-colored multigraph with colors $\{G^{\mathcal{P}}_1, \ldots, G^{\mathcal{P}}_k\}$ and vertex set $V(G)$, in which for each $ij\in {[m]\choose 2}$ and $\rho\in [k]$, the bipartite graph $G^{\mathcal{P}}_{\rho}[V_i,V_j]$ is a complete bipartite graph if $ij\in E(R_{\rho})$, and an empty bipartite graph if $ij\notin E(R_{\rho})$, and there is no edge within each part.

\begin{lemma}[Multicolor Embedding Lemma \cite{CKLLS,KoSi}]\label{le:embedding}
Suppose $0 \leq \frac{1}{n} \ll \varepsilon \ll \gamma, \frac{1}{h}\leq 1$.
Let $H$ be an $r$-vertex $h$-edge multigraph, and $G$ be a simply $k$-nested-colored multigraph with an $\varepsilon$-regular partition $\mathcal{P}$.
If $G^{\mathcal{P}}(\varepsilon, \gamma)$ contains a multicolored copy of $H$, then $G$ contains a multicolored copy of $H$.
\end{lemma}

Now we prove a stability result for color-critical simple graphs with any number of vertices.

\begin{lemma}\label{le:any_vertex-stability}
Let $r\geq 5$, $k> \max\big\{\frac{r-1}{r-2}(h-1), 2\frac{r-1}{r}\big(h-\frac{r(r-2)}{4}+1\big)\big\}$ and $0<\frac{1}{n}\ll \eta \ll \mu, \frac{1}{k}<1$.
Let $H$ be an $r$-color-critical graph with $h$ edges.
Let $G$ be an $n$-vertex simply $k$-nested-colored multicolored-$H$-free multigraph with $e(G)\geq k\cdot t_{r-1}(n)-\eta n^2$.
Then $|G\triangle (k\cdot T_{r-1}(n))|\leq \mu n^2$.
\end{lemma}

\begin{proof}
Let $0<\frac{1}{n}\ll \varepsilon \ll \gamma \ll \eta \ll \eta' \ll \mu, \frac{1}{k}<1$.
Let $H_c$ be the color-reduced multigraph of $H$.
Applying Lemma~\ref{le:regulartity} to $G$ with the constants $\varepsilon, k, \frac{1}{\varepsilon}$ playing the roles of $\varepsilon, k, M_0$, we obtain an $M$ and an $\varepsilon$-regular partition $\mathcal{P}=\{V_1, \ldots, V_m\}$ with $\frac{1}{\varepsilon}\leq m\leq M$, and we may assume that $n\gg M$.
Let $R$ be the $(\varepsilon, \gamma, \mathcal{P})$-reduced multigraph of $G$.
Let $\{G_1, \ldots, G_k\}$ be the colors of $G$.
Since $G$ is nested, we may assume that $G_1 \subseteq \cdots \subseteq G_k$ without loss of generality.
For each $1\leq \rho_1 <\rho_2 \leq k$ and $ij\in {[m]\choose 2}$, if $(V_i,V_j)$ is $(\varepsilon, \gamma; \rho_1)$-lower-regular, then $(V_i,V_j)$ is also $(\varepsilon, \gamma; \rho_2)$-lower-regular since $G_{\rho_1}\subseteq G_{\rho_2}$.
This implies that if $ij\in E(R_{\rho_1})$, then we also have $ij\in E(R_{\rho_2})$.
Hence, $R$ is nested, and therefore $G^{\mathcal{P}}= G^{\mathcal{P}}(\varepsilon, \gamma)$ is nested by the definition of $G^{\mathcal{P}}$.
Since $G$ is multicolored-$H$-free, the multigraph $G^{\mathcal{P}}$ is multicolored-$H$-free by Lemma~\ref{le:embedding}.
Then $R$ is multicolored-$H_c$-free, since otherwise $G^{\mathcal{P}}$ contains a multicolored copy of $H$.

By Lemma~\ref{le:regulartity}, we have $d(R)\geq d(G)-2(\varepsilon + \gamma)$.
Combining with $0<\frac{1}{n}\ll \varepsilon \ll \gamma \ll \eta$, we have
\begin{align*}
e(R) = &~d(R)m^2 \geq d(G)m^2-2(\varepsilon + \gamma)m^2 \\
\geq &~\left(k\cdot t_{r-1}(n)-\eta n^2\right)\frac{m^2}{n^2}-2(\varepsilon + \gamma)m^2 ~\geq k\cdot t_{r-1}(m)-2\eta m^2.
\end{align*}
Note that $|V(R)|=m\geq \frac{1}{\varepsilon}$ and $0<\varepsilon \ll \eta \ll \eta' \ll \frac{1}{k}<1$.
Then by Lemma~\ref{le:r_vertex-stability}, we have $|R\triangle (k\cdot T_{r-1}(m))|\leq \eta' m^2$.
Thus we can choose a copy $A$ of $k\cdot T_{r-1}(m)$ on vertex set $V(R)=[m]$ such that $|R\triangle A|\leq \eta' m^2$.
Let $U_1, \ldots, U_{r-1}$ be the partite sets of $A$.
For each $i\in [r-1]$, let $W_i=\bigcup_{j\in U_i}V_j$.
Let $B$ be the complete $(r-1)$-partite graph with partite sets $W_1, \ldots, W_{r-1}$ in which every edge has multiplicity $k$.
Since $|R\triangle A|\leq \eta' m^2$, we have $|G^{\mathcal{P}}\triangle B|\leq \eta' m^2\left\lceil\frac{n}{m}\right\rceil^2< 2\eta' n^2$.
Note that $\big\lfloor\frac{m}{r-1}\big\rfloor \leq |U_{i_1}|\leq \big\lceil\frac{m}{r-1}\big\rceil$ for each $i_1\in [r-1]$, and $\big\lfloor\frac{n}{m}\big\rfloor \leq |V_{i_2}|\leq \big\lceil\frac{n}{m}\big\rceil$ for each $i_2\in [m]$.
Thus for each $1\leq i<j\leq r-1$, we have $\left||W_i|-|W_j|\right|\leq \left\lceil\frac{n}{m}\right\rceil+\big\lfloor\frac{m}{r-1}\big\rfloor < 2\frac{n}{m}$.
Thus we can obtain a copy of $k\cdot T_{r-1}(n)$ by deleting or adding at most $k(r-1)(2\frac{n}{m}) n$ edge from $B$.
So $|B\triangle (k\cdot T_{r-1}(n))|\leq k(r-1)(2\frac{n}{m}) n<2kr\frac{1}{m}n^2\leq 2kr\varepsilon n^2 \leq \eta' n^2$.
Then $$|G^{\mathcal{P}}\triangle (k\cdot T_{r-1}(n))|\leq |G^{\mathcal{P}}\triangle B|+ |B\triangle (k\cdot T_{r-1}(n))|< 2\eta' n^2 + \eta' n^2 \leq 3\eta' n^2,$$
and thus $$e(G^{\mathcal{P}})\leq k\cdot t_{r-1}(n)+3\eta' n^2 \leq e(G)+ 4\eta' n^2.$$
From the definition of $G^{\mathcal{P}}$, we know that there exists an edge subset $E\subseteq E(G)$ with
$$|E|\leq k\left(\varepsilon m^2 \left(\frac{n}{m}\right)^2 + \frac{1}{m}\cdot \frac{n^2}{2} + \gamma n^2\right)=k\left(\varepsilon + \frac{1}{2m} + \gamma\right)n^2 \leq \eta' n^2$$
such that $E(G)\setminus E \subseteq E(G^{\mathcal{P}})$.
Then $$|G\triangle G^{\mathcal{P}}|\leq |E|+\left(e(G^{\mathcal{P}})-|E(G)\setminus E|\right)\leq 2|E|+4\eta' n^2\leq 6\eta' n^2.$$
Thus $|G\triangle (k\cdot T_{r-1}(n))|\leq |G\triangle G^{\mathcal{P}}|+|G^{\mathcal{P}}\triangle (k\cdot t_{r-1}(n))|\leq 9\eta' n^2\leq \mu n^2$.
\end{proof}

Now we have all the ingredients to state and prove our main result.
Note that Theorem~\ref{th:main} follows from Theorem~\ref{th:main+} below immediately since $2\frac{r-1}{r}(h-1)>\max\big\{\frac{r-1}{r-2}(h-1), 2\frac{r-1}{r}\big(h-\frac{r(r-2)}{4}+1\big)\big\}$ for $r\geq 5$.

\begin{theorem}\label{th:main+}
Let $r\geq 5$ and $H$ be an $r$-color-critical graph with $h$ edges.
If $n$ is sufficiently large and $$k> \max\left\{\frac{r-1}{r-2}(h-1), 2\frac{r-1}{r}\left(h-\frac{r(r-2)}{4}+1\right)\right\},$$ then $\ex_k(n, H)=k\cdot t_{r-1}(n)$, and the unique $n$-vertex $k$-color extremal multigraph of $H$ consists of $k$ colors all of which are identical copies of $T_{r-1}(n)$.
\end{theorem}

\begin{proof}
Let $0<\frac{1}{n}\ll \mu \ll \frac{1}{h} < 1$.
We may assume that $H$ contains no isolated vertex, so $|V(H)|\leq 2h$.
Let $G$ be an $n$-vertex simply $k$-colored multicolored-$H$-free multigraph with $e(G)\geq k\cdot t_{r-1}(n)$.
We shall show that $G=k\cdot T_{r-1}(n)$.
By Lemma~\ref{le:nested_degree}, we may assume that $G$ is nested and $\delta(G)\geq k\delta(T_{r-1}(n))$.
By Lemma~\ref{le:any_vertex-stability}, we have $|G\triangle (k\cdot T_{r-1}(n))|\leq \mu n^2$.

Let $K^{(h)}$ be a complete $(r-1)$-partite multigraph with $2h$ vertices in each part whose edges all have multiplicity at least $h$.
Let $K'$ be the underlying graph of $K^{(h)}$, i.e., $K'$ is a complete $(r-1)$-partite graph with $2h$ vertices in each part.
Let $X\colonequals \big\{e\in{V(G)\choose 2}\colon\, w_G(e)=k\big\}$.
Since $|G\triangle (k\cdot T_{r-1}(n))|\leq \mu n^2$, we have $|X|\geq t_{r-1}(n)-\mu n^2 > t_{r-2}(n)+\frac{1}{2r^2}n^2 > \ex(n, K')$ by the Erd\H{o}s-Stone-Simonovits theorem.
Thus $G$ contains $K^{(h)}$ as a subgraph.
Let $K$ be a copy of $K^{(h)}$ in $G$ with partite sets $W_1, \ldots, W_{r-1}$.

\begin{claim}\label{cl:K} $K$ is $H$-friendly.
\end{claim}

\begin{proof} Let $K^{+}$ be the multigraph obtained by adding a new vertex $v$ to $K$ and adding edges incident to $v$ of multiplicity at most $k$ such that $d_{K^+}(v)\geq (r-2)2hk$ and $d_{W_i}(v)\geq 1$ for all $i\in [r-1]$.
Without loss of generality, we may assume that $d_{W_1}(v)=\min_{i\in [r-1]} d_{W_i}(v)$, and let $w\in W_1$ with $w_{K^+}(vw)\geq 1$.
We claim that $d_{W_i}(v)\geq kh$ for each $i\geq 2$.
Indeed, if $\min_{2\leq i\leq r-1} d_{W_i}(v)<kh$, then $d_{K^+}(v)<2kh+k(r-3)2h =(r-2)2hk$, a contradiction.

Let $ab$ be a critical edge of $H$ with $d_{H}(a)\leq d_H(b)$.
Since $\chi(H)=r$, we have $\chi(H-\{a,b\})\geq r-2$.
Thus $h\geq {r\choose 2}$ and $e(H-\{a,b\})\geq {r-2 \choose 2}$.
Then $d_H(a)+d_H(b)-1+{r-2 \choose 2}\leq h$, so $d_H(a)\leq \frac{h}{2}-1$.
Let $d\colonequals d_H(a)-1$.

For each $i\in \{2, \ldots, r-1\}$, let $w_{i,1}, \ldots, w_{i,2h}$ be the vertices of $W_i$ with $w_{K^+}(vw_{i,1})\geq \cdots \geq w_{K^+}(vw_{i,2h})$.
We claim that for each $i\in \{2, \ldots, r-1\}$ and $j\in [d]$, we have $w_{K^+}(vw_{i,j})\geq d-j+2$.
To see this, note that $$w_{K^+}(vw_{i,j})\geq \frac{1}{2h-j+1}\left(d_{W_i}(v)-\sum\nolimits_{\ell =1}^{j-1}w_{K^+}(vw_{i,\ell})\right)\geq \frac{1}{2h-j+1}(kh-k(j-1))$$ and
\begin{align*}
~&~kh-k(j-1)-(2h-j+1)(d-j+2) \\
\geq &~kh-k(j-1)-(2h-j+1)\left(\frac{h}{2}-2-j+2\right) \\
= &~-j^2-\left(k-\frac{5}{2}h-1\right)j+ k(h+1)-h^2-\frac{h}{2}.
\end{align*}
Let $f(x)\colonequals -x^2-(k-\frac{5}{2}h-1)x+ k(h+1)-h^2-\frac{h}{2}$.
Then $\min\{f(x)\colon\, 1\leq x\leq d\}\geq \min\{f(x)\colon\, 1\leq x\leq \frac{h}{2}\}\geq \min\{f(1), f(\frac{h}{2})\}\geq 0$.
Thus $w_{K^+}(vw_{i,j})\geq d-j+2$ for each $i\in \{2, \ldots, r-1\}$ and $j\in [d]$.

Consider a critical coloring of $H$ with color classes $V_1, \ldots, V_r$, where $a\in V_r$, $b\in V_1$ and $e(V_1, V_r)=1$.
Since $e(V_1, V_r)=1$, we may further assume that $V_r=\{a\}$ by putting the vertices of $V_r\setminus\{a\}$ into $V_1$ if necessary.
For each $2\leq i\leq r-1$, let $V_i'=N_H(a)\cap V_i$.
Take an embedding $\phi: V(H) \to V(K^+)$ and an ordering $(e_1, \ldots, e_h)$ of the edges of $H$ with $e_1=ab$ and $\{e_2, \ldots, e_{d+1}\}=\{au\colon\, u\in N_H(a)\setminus \{b\}\}$ such that
\begin{itemize}
\item[(i)] $\phi(a)=v$, $\phi(b)=w$, $\phi(V_1)\subseteq W_1$,
\item[(ii)] $\phi(V_i')=\{w_{i,1}, \ldots, w_{i,|V_i'|}\}$ and $\phi(V_i)\subseteq W_i$ for each $2\leq i\leq r-1$, and
\item[(iii)] $w_{K^+}(\phi(e_2))\leq \cdots \leq w_{K^+}(\phi(e_{d+1}))$.
\end{itemize}
Then for each $2\leq j\leq d+1$ and $2\leq i\leq r-1$,
since $w_{K^+}(\phi(e_j))=\max_{d+2-j}\{w_{K^+}(\phi(e_2)), \ldots,$  $w_{K^+}(\phi(e_{d+1}))\}$ and $w_{K^+}(vw_{i,d+2-j})=\max_{d+2-j}\{w_{K^+}(vw_{i,1}), \ldots, w_{K^+}(vw_{i,2h})\}$,
we have $w_{K^+}(\phi(e_j))\geq w_{K^+}(vw_{i,d+2-j})\geq j=\sum_{i=1}^{j}w_H(e_i)$.
Thus $w_{K^+}(\phi(e_j))\geq \sum_{i=1}^{j}w_H(e_i)$ for each $1\leq j\leq h$.
By Lemma~\ref{le:Hall}, $K^+$ contains a multicolored copy of $H$.
Thus $K$ is $H$-friendly.
\end{proof}

By Claim~\ref{cl:K}, $K$ is $H$-friendly.
Moreover, $K$ is an induced subgraph of $G$, since any additional edge within a color class of $K$ yields a multicolored copy of $H$.
By Lemma~\ref{le:partite}, $G$ is $(r-1)$-partite.
Since $T_{r-1}(n)$ is the unique $(r-1)$-partite graph on $n$ vertices with $t_{r-1}(n)$ edges, we have $e(G)\leq k\cdot t_{r-1}(n)$.
Thus $e(G)= k\cdot t_{r-1}(n)$, and $G=k\cdot T_{r-1}(n)$.
This completes the proof of Theorem~\ref{th:main+}.
\end{proof}

\section{Concluding remarks}
\label{sec:ch-conclu}

In this paper, we prove that Conjecture~\ref{conj:KSSV-critical} holds for $k\geq 2\frac{r-1}{r}(h-1)$.
We now give a remark regarding the challenges in employing our current arguments to fully resolve Conjecture~\ref{conj:KSSV-critical}.
The main challenge in settling the full upper range $k\geq \frac{r-1}{r-2}(h-1)$ lies in improving the lower bound on $k$ in Lemma~\ref{le:r_vertex-degree}.
Essentially, the proof of this key lemma is completed in two steps:
\begin{itemize}
\item[(1)] we first find a multigraph $F$ on $r-1$ vertices in $G$ via Tur\'{a}n's theorem and Lemma~\ref{le:vertices}, and then obtain a multigraph $F^+$ by adding a new vertex and several new edges to $F$;
\item[(2)] we show that $F^+$ contains a multicolored copy of $H$ utilizing certain packing arguments, which implies that $F$ is $H$-friendly and therefore $G$ is $(r-1)$-partite by Lemma~\ref{le:partite}.
\end{itemize}
In the first step, the main task is to identify a complete multigraph on $r-s$ vertices with edges of high multiplicities, where $s= \left\lceil\frac{r}{2}\right\rceil$.
If our focus is solely on this particular task, then as the value of $s$ increases, the required value of $k$ decreases accordingly.
In the second step, the main task is to show that there is a packing of two graphs.
In this step, if the value of $s$ increases, then the required multiplicity on the edges in $F^+$ must increase, and the required value of $k$ increases accordingly.
Therefore, there is a trade-off between the two steps.
After conducting a thorough trade-off analysis, we have selected $s= \left\lceil\frac{r}{2}\right\rceil$ and $k\geq 2\frac{r-1}{r}(h-1)$ as the optimal parameter values.
Although in the small case $r=5$, we can improve the bound on $k$ from $\frac{8}{5}(h-1)$ to $\frac{3}{2}(h-1)$ using some complicated analysis (which we decide to not include here),
the packing method presented here cannot be used to significantly improve the lower bound on $k$ for general $r$.
Hence, some novel ideas for embedding are called for.

We next provide some remarks on our new arguments compared to the previous papers \cite{CKLLS,KSSV} on the same topic.
Note that the graphs considered in \cite{CKLLS,KSSV} possess some nice properties.
Thus by some clever arguments, they can find a multicolored $H$ in a multigraph with one edge of high multiplicity.
However, when dealing with general color-critical (multi)graphs, the presence of merely a single edge with high multiplicity is insufficient for us to identify a multicolored $H$.
We must find a complete multigraph on $r-s$ vertices with edges of high multiplicities, where $s= \left\lceil\frac{r}{2}\right\rceil$ by conducting a trade-off analysis as mentioned above.
And finally, we use a novel graph packing technique to find the required multicolored copy of $H$.

Let $H$ be a non-color-critical graph with $h$ edges and $\chi(H)\geq 3$.
Chakraborti, Kim, Lee, Liu and Seo (see \cite[Proposition~1.6]{CKLLS}) proved that $\ex_k(n, H)> \max\big\{(h-1){n\choose 2}, k\cdot \ex(n, H)\big\}$ when $k\geq \frac{r-1}{r-2}(h-1)$ and $n$ is sufficiently large.
Thus Conjecture~\ref{conj:KSSV-critical} does not hold for general $h$-edge graphs with chromatic number $r$ when $k\geq \frac{r-1}{r-2}(h-1)$.
However, we can prove the following result for any general graph (not necessarily color-critical) with chromatic number at least three,
showing that this conjecture stays true in the range $h\leq k \leq h+\big\lfloor\frac{r}{2}\big\rfloor-1$.

\begin{theorem}\label{th:ksmall}
For any graph $H$ with $h$ edges and $\chi(H)=r\geq 3$, if $n$ is sufficiently large and $$h\leq k \leq h+\left\lfloor\frac{r}{2}\right\rfloor-1,$$ then $\ex_k(n, H)=(h-1){n\choose 2}$ and the unique $n$-vertex $k$-color extremal multigraph of $H$ consists of exactly $h-1$ colors each of which is a copy of $K_n$.
\end{theorem}

We supply a proof sketch in Appendix~\ref{ap:2}.
It would be interesting to determine the maximum $k$ such that $\ex_k(n, H)=(h-1){n\choose 2}$ for every graph $H$ with $h$ edges and $\chi(H)\geq 3$.

Besides the sum condition $\sum_{i\in [k]}e(G_i)$, the product condition $\prod_{i\in [k]}e(G_i)$ for multicolored triangles was also studied in \cite{FaMR24+Combin,Fra22+,HFGLSTVZ}.
Further research on the product condition for general cliques or color-critical graphs could be fruitful.
On the other hand, it is also natural to study multicolored spanning structures instead of a fixed small subgraph, such as Hamilton paths and cycles \cite{ADGSMS,JoKi}, perfect matchings \cite{CHWW,JoKi,LuWY23JCTB}, and graph factors \cite{ABHP,CHWW}.
\vspace{-0.2cm}

\section*{Acknowledgement}

The authors are grateful to the anonymous referees for valuable comments and suggestions which improved the presentation of this paper.
Jie Ma is supported by National Key Research and Development Program of China (Grant No. 2023YFA1010201) and National Natural Science Foundation of China (Grant No. 12125106).
Xihe Li is supported by the National Natural Science Foundation of China (Grant No. 12501492), Shaanxi Province Postdoctoral Science Foundation (Grant No. 2024BSHSDZZ155) and the Fundamental Research Funds for the Central Universities (Grant No. GK202506024).
We would like to thank Dilong Yang for helpful discussions.

\begin{spacing}{0.8}

\end{spacing}

\vspace{-0.5cm}

\appendix

\section*{Appendix}

\section{Proof of Lemmas~\ref{le:calculation-1} and \ref{le:calculation-2}}\label{ap:1}

For convenience, we will restate Lemmas~\ref{le:calculation-1} and \ref{le:calculation-2}.

\medskip\noindent{\bf Lemma~\ref{le:calculation-1}.} {\it Let $r\geq 5$, $s=\left\lceil\frac{r}{2}\right\rceil$ and $h\geq {r\choose 2}$.
If $\ell$ and $i$ are two integers satisfying one of the following statements:
\begin{itemize}
\item[{\rm (i)}] $2\leq \ell \leq s-2$ and $r-s+1 \leq i\leq r-1$, or
\item[{\rm (ii)}] $\ell = s-1$ and $r-s+2 \leq i\leq r-1$,
\end{itemize}
then $\frac{2}{r}\big(r-2-\frac{\ell-1}{i-\ell}\big)\big(h-\frac{r(r-2)}{4}+1\big)\geq h-((\ell-1)s+r-i).$}
\vspace{0.05cm}

\begin{proof}
Let $f(r,i,\ell)\colonequals \frac{2}{r}\big(r-2-\frac{\ell-1}{i-\ell}\big)\big(h-\frac{r(r-2)}{4}+1\big) - h+(\ell-1)s+r-i.$
We shall show that $f(r,i,\ell)\geq 0$.
Note that in both cases we have $i-\ell\geq r-2s+3$.
Since $h\geq {r\choose 2}$ and
\begin{align*}
 \frac{2}{r}\left(r-2-\frac{\ell -1}{i-\ell}\right)-1 = &~1-\frac{4}{r}-\frac{2}{r}\cdot \frac{\ell -1}{i-\ell} ~\geq 1-\frac{4}{r}-\frac{2}{r}\cdot\frac{s-2}{r-2s+3} \\
 \geq &~1-\frac{4}{r}-\frac{2}{r}\cdot \frac{(r+1)/2-2}{r-(r+1)+3} ~= \frac{r-5}{2r} ~\geq 0,
\end{align*}
we have
\begin{align*}
 f(r,i,\ell) \geq &~\frac{2}{r}\left(r-2-\frac{\ell -1}{i-\ell }\right)\left({r \choose 2}-\frac{r(r-2)}{4}+1\right)-{r \choose 2}+(\ell-1)s+r-i \\
 = &~-\frac{r}{2}\left(1+\frac{\ell-1}{i-\ell}\right)+(\ell-1)s+r-i+\frac{2}{r}\left(r-2-\frac{\ell -1}{i-\ell }\right) \\
 \geq &~-\frac{r}{2}\left(1+\frac{\ell-1}{i-\ell}\right)+(\ell-1)s+r-i+1.
\end{align*}
If $\ell \geq 3$, then
\begin{align*}
 f(r,i,\ell) \geq &~-\frac{r}{2}\left(1+\frac{\ell-1}{i-\ell}\right)+(\ell-1)s+r-i+1 \\
 \geq &~-\frac{r}{2}\left(1+\frac{\ell-1}{r-2s+3}\right)+(\ell-1)\frac{r}{2}+r-(r-1)+1 \\
 \geq &~-\frac{r}{2}\left(2+\frac{\ell-1}{r-(r+1)+3}-\ell\right)+2 ~= \frac{r}{2}\cdot \frac{\ell-3}{2}+2 ~> 0.
\end{align*}
If $\ell =2$, then
\begin{align*}
 f(r,i,\ell) \geq &~-\frac{r}{2}\left(1+\frac{\ell-1}{i-\ell}\right)+(\ell-1)s+r-i+1 \\
 \geq &~-\frac{r}{2}\left(1+\frac{1}{i-2}\right)+\frac{r}{2}+r-i+1 \\
 = &~\frac{1}{2(i-2)}\left(-2i^2 + (2r+6)i -5r -4 \right) \\
 \geq &~\frac{1}{2(i-2)}\min\bigg\{-2\bigg(\frac{r+1}{2}\bigg)^2 + (2r+6)\frac{r+1}{2} -5r -4, \\
 &~\hspace{2.45cm}  -2(r-1)^2 + (2r+6)(r-1) -5r -4\bigg\} \\
 = &~\frac{1}{4(i-2)}\min\left\{r^2-4r-3, 2(3r-12)\right\} ~> 0,
\end{align*}
where the third inequality holds since $i\geq r-s+1\geq \frac{r+1}{2}$ and $i\leq r-1$.
This completes the proof.
\end{proof}

\medskip\noindent{\bf Lemma~\ref{le:calculation-2}.} {\it Let $r\geq 5$, $s=\left\lceil\frac{r}{2}\right\rceil$, $2\leq \ell \leq s-1$ and $h\geq {r\choose 2}$. Then $$2\frac{r-\ell-1}{r-\ell}\cdot \frac{r-1}{r}\left(h-\frac{r(r-2)}{4}+1\right)\geq h-(\ell-1)s.$$}
\vspace{-0.4cm}

\begin{proof}
Since $\ell \leq s-1\leq (r-1)/2$, we have
$$2\frac{r-\ell-1}{r-\ell}\cdot \frac{r-1}{r}-1\geq 2\frac{r-1-(r-1)/2}{r-(r-1)/2}\cdot \frac{r-1}{r}-1= \frac{r^2-5r+2}{r(r+1)}>0.$$
Thus
\begin{align*}
 &~2\frac{r-\ell-1}{r-\ell}\cdot\frac{r-1}{r}\left(h-\frac{r(r-2)}{4}+1\right)-h+(\ell -1)s \\
 \geq &~\frac{2(r-\ell -1)(r-1)}{r(r-\ell )}\left({r \choose 2}-\frac{r(r-2)}{4}+1\right)-{r \choose 2}+(\ell -1)s \\
 \geq &~\frac{2(r-\ell-1)(r-1)}{r(r-\ell)}\cdot \frac{r^2+4}{4}-{r\choose 2}+ (\ell -1)\frac{r}{2} \\
 \geq &~\frac{1}{2r(r-\ell)}\left(-r^2\ell ^2+(r^3+r^2-4r+4)\ell -2r^3+5r^2-8r+4\right) \\
 \geq &~\frac{1}{2r(r-\ell)}\min\bigg\{-r^2\cdot 2^2+(r^3+r^2-4r+4)\cdot 2 -2r^3+5r^2-8r+4, \\
 &~\hspace{2.45cm}  -r^2\cdot \bigg(\frac{r-1}{2}\bigg)^2+(r^3+r^2-4r+4)\cdot \bigg(\frac{r-1}{2}\bigg) -2r^3+5r^2-8r+4\bigg\} \\
 \geq &~\frac{1}{8r(r-\ell)}\min\left\{4(3r^2-16r+12), r^4-6r^3+9r^2-16r+8\right\}.
\end{align*}
Since $r\geq 5$, we have $3r^2-16r+12>0$.
If $r\geq 6$, then $r^4-6r^3\geq 0$ and $9r^2-16r\geq 0$, so $r^4-6r^3+9r^2-16r+8> 0$.
If $r=5$, then $r^4-6r^3+9r^2-16r+8=5(125-150+45-16)+8>0$.
This completes the proof.
\end{proof}

\section{Proof of Theorem~\ref{th:ksmall}}\label{ap:2}

The proof of Theorem~\ref{th:ksmall} follows from the following four lemmas.
Lemma~\ref{le:ksmall_r_vertex} can be proved using Lemma~\ref{le:ksmall_r_vertex-degree} and a deleting argument.
The proofs of Lemmas~\ref{le:ksmall_r_vertex-stability}, \ref{le:ksmall_any_vertex-stability} and the final proof of Theorem~\ref{th:ksmall} can be completed by a routine application of the stability argument introduced in \cite{CKLLS}, so we omit the details and we only supply a proof of Lemma~\ref{le:ksmall_r_vertex-degree}.

\begin{lemma}\label{le:ksmall_r_vertex-degree}
Let $r\geq 3$, $h\leq k\leq h+\big\lfloor\frac{r}{2}\big\rfloor-1$ and $0 \ll \frac{1}{n} \ll \delta \ll \frac{1}{k}<1$.
Let $H$ be an $r$-vertex $h$-edge multigraph with $w_H(e)\geq 1$ for each $e\in {V(H)\choose 2}$, and $G$ be an $n$-vertex simply $k$-colored multicolored-$H$-free multigraph with $\delta(G)\geq (1-\delta)(h-1)(n-1)$.
Then $w_G(e)\leq h-1$ for each $e\in {V(G)\choose 2}$.
\end{lemma}

\begin{lemma}\label{le:ksmall_r_vertex}
Let $r\geq 3$, $h\leq k\leq h+\big\lfloor\frac{r}{2}\big\rfloor-1$, $n$ be sufficiently large, and $H$ be an $r$-vertex $h$-edge multigraph with $w_H(e)\geq 1$ for each $e\in {V(H)\choose 2}$.
Then $\ex_k(n, H)=(h-1){n\choose 2}$, and the $n$-vertex $k$-color extremal multigraph of $H$ consists of exactly $h-1$ colors each of which is a copy of $K_n$.
\end{lemma}

\begin{lemma}\label{le:ksmall_r_vertex-stability}
Let $r\geq 3$, $h\leq k\leq h+\big\lfloor\frac{r}{2}\big\rfloor-1$ and $0<\frac{1}{n}\ll \eta \ll \varepsilon, \frac{1}{k}<1$.
Let $H$ be an $r$-vertex $h$-edge multigraph with $w_H(e)\geq 1$ for each $e\in {V(H)\choose 2}$.
Let $G$ be an $n$-vertex simply $k$-nested-colored multicolored-$H$-free multigraph with $e(G)\geq (h-1){n\choose 2}-\eta n^2$.
Then $|G\triangle ((h-1)\cdot K_n)|\leq \varepsilon n^2$.
\end{lemma}

\begin{lemma}\label{le:ksmall_any_vertex-stability}
Let $r\geq 3$, $h\leq k\leq h+\big\lfloor\frac{r}{2}\big\rfloor-1$ and $0<\frac{1}{n}\ll \eta \ll \mu, \frac{1}{k}<1$.
Let $H$ be an $h$-edge graph with $\chi(H)=r$.
Let $G$ be an $n$-vertex simply $k$-nested-colored multicolored-$H$-free multigraph with $e(G)\geq (h-1){n\choose 2}-\eta n^2$.
Then $|G\triangle ((h-1)\cdot K_n)|\leq \mu n^2$.
\end{lemma}

\begin{proof}[Proof of Lemma~\ref{le:ksmall_r_vertex-degree}]
Suppose for a contradiction that $G$ contains an edge $v_1v_2$ of multiplicity at least $h$.
Applying Lemma~\ref{le:vertices} iteratively, we can find vertices $v_{3}, \ldots, v_{r}\in V(G)\setminus \{v_1, v_2\}$ such that for each $3 \leq i \leq r$, we have $\sum_{j\in [i-1]} w_G(v_iv_j)\geq (h-1)(i-1).$
For each $3\leq i\leq r$, we denote the edges in $\{v_iv_{\ell}\colon\, 1\leq \ell \leq i-1\}$ by $e_{i,1}, \ldots, e_{i,i-1}$ with $w_G(e_{i,1})\geq \cdots \geq w_G(e_{i,i-1})$.
Then for each $3\leq i\leq r$ and $1\leq \ell \leq i-1$, we have
\begin{align*}
w_G(e_{i, \ell})\geq &~\frac{1}{(i-1)-(\ell-1)}\left(\sum\nolimits_{j\in [i-1]} w_G(v_iv_j)-(\ell-1)k\right) \\
 \geq &~\frac{1}{i-\ell}\left((h-1)(i-1)-(\ell-1)\left(h+\left\lfloor\frac{r}{2}\right\rfloor-1\right)\right) \\
 = &~h-\frac{1}{i-\ell}\left(i-1+\left(\left\lfloor\frac{r}{2}\right\rfloor-1\right)(\ell-1)\right).
\end{align*}
In particular, we have $w_G(e_{i,1})\geq h-1$ for all $3\leq i\leq r$.

We use $e_1, e_2, \ldots, e_{{r\choose 2}}$ to denote the edges
$$v_1v_2, ~ e_{r,1}, \ldots, e_{3,1}, ~ e_{r,2}, \ldots, e_{3,2}, ~ e_{r,3}, \ldots, e_{4,3}, ~ \ldots, ~ e_{r,r-2}, e_{r-1,r-2}, ~ e_{r,r-1},$$
respectively.
We shall show that $w_{G}(e_j)\geq h-(j-1)$ for all $1\leq j\leq {r\choose 2}$.
This holds clearly for $1\leq j\leq r-1$ since $e_1=v_1v_2$ and $w_G(e_{i,1})\geq h-1$ for all $3\leq i\leq r$.
Next we consider $r\leq j\leq {r\choose 2}$, that is, we consider $e_{i,\ell}$ with $3\leq i\leq r$ and $2\leq \ell \leq i-1$.
Now $e_{i,\ell} = e_j$ with $j=r-1+\sum_{t=2}^{\ell} (r-t) - (i-\ell-1)=r-i+\ell+\frac{1}{2}(\ell-1)(2r-\ell-2)$.
It suffices to show that $h-\big(r-i+\ell+\frac{1}{2}(\ell-1)(2r-\ell-2)-1\big)\leq h-\frac{1}{i-\ell}\left(i-1+\left(\left\lfloor\frac{r}{2}\right\rfloor-1\right)(\ell-1)\right).$
If $2\leq \ell \leq i-2$, then
\begin{align*}
~ &~r-i+\ell+\frac{1}{2}(\ell-1)(2r-\ell-2)-1-\frac{1}{i-\ell}\left(i-1+\left(\left\lfloor\frac{r}{2}\right\rfloor-1\right)(\ell-1)\right) \\
\geq &~r-i+\ell+\frac{1}{2}(\ell-1)(2r-\ell-2)-1-\frac{1}{2}\left(i-1+\left(\frac{r}{2}-1\right)(\ell-1)\right) \\
= &~\frac{1}{2}(\ell-1)\left(\frac{3}{2}r-\ell-1\right) + r-\frac{3}{2}i+\ell-\frac{1}{2} \\
\geq &~\frac{1}{2}\left(\frac{3}{2}r-\ell-1\right) + r-\frac{3}{2}i+\ell-\frac{1}{2} ~= \frac{7}{4}r-\frac{3}{2}i +\frac{\ell}{2}-1 ~> 0.
\end{align*}
If $\ell=i-1$, then
\begin{align*}
~ &~r-i+\ell+\frac{1}{2}(\ell-1)(2r-\ell-2)-1-\frac{1}{i-\ell}\left(i-1+\left(\left\lfloor\frac{r}{2}\right\rfloor-1\right)(\ell-1)\right) \\
= &~r-i+\ell+\frac{1}{2}(\ell-1)(2r-\ell-2)-1-\left(i-1+\left(\left\lfloor\frac{r}{2}\right\rfloor-1\right)(\ell-1)\right) \\
= &~(\ell-1)\left(\left\lceil\frac{r}{2}\right\rceil-\frac{\ell}{2}\right)+r-2i+\ell ~= (i-2)\left(\left\lceil\frac{r}{2}\right\rceil-\frac{i-1}{2}\right)+r-i-1\\
\geq &~(i-2)\left(\left\lceil\frac{r}{2}\right\rceil-\frac{r-1}{2}\right)+r-i-1 ~\geq 0.
\end{align*}
Thus $w_{G}(e_j)\geq h-(j-1)$ for all $1\leq j\leq {r\choose 2}$.
For each $1\leq j\leq {r\choose 2}$, let $e_j'\colonequals e_{{r\choose 2}+1-j}$, so $w_{G}(e_j')\geq h-({r\choose 2}-j)$.
Since $w_H(e)\geq 1$ for each $e\in {V(H)\choose 2}$, we can deduce that for any $E\subseteq {V(H)\choose 2}$ with $|E|=j$, we have $\sum_{e\in E}w_H(e)\leq h- ({r\choose 2}-j)\leq w_{G}(e_j')$.
Therefore, $G[\{v_1, v_2, \ldots, v_r\}]$ contains a multicolored copy of $H$ by Lemma~\ref{le:Hall}, a contradiction.
This completes the proof of Lemma~\ref{le:ksmall_r_vertex-degree}.
\end{proof}

\vspace{0.5cm}

\noindent{\it E-mail address}: xiheli@snnu.edu.cn ~~(X. Li)
\vspace{0.1cm}

\noindent{\it E-mail address}: jiema@ustc.edu.cn ~~(J. Ma)
\vspace{0.1cm}

\noindent{\it E-mail address}: lhotse@mail.ustc.edu.cn ~~(Z. Zheng)

\end{document}